\documentclass[11pt]{amsart}
\usepackage[top=1in,
 bottom=1in,
 left=1.25in,
 right=1.25in, includehead]{geometry}
\usepackage{amsmath, amssymb, amsthm}
\usepackage{mathrsfs}
\usepackage{enumerate}
\usepackage{hyperref}
\usepackage{xcolor}
\usepackage{bbm}
\usepackage[capitalize]{cleveref}
\usepackage{enumerate}
\usepackage{cancel}
\usepackage[backend=biber, style=alphabetic, giveninits=true, maxalphanames=4,minalphanames=4, maxnames=99]{biblatex}
\usepackage{marginnote}
\usepackage[normalem]{ulem}
\usepackage{comment}
\newtheorem{thm}{Theorem}[section]
\newtheorem{cor}[thm]{Corollary}
\newtheorem{prop}[thm]{Proposition}
\newtheorem{lem}[thm]{Lemma}

\newtheorem{ques}[thm]{Question}

\usepackage{cancel}
\usepackage{chngcntr}
\theoremstyle{definition}
\newtheorem{defn}[thm]{Definition}

\newtheorem{exmp}[thm]{Example}

\newtheorem{prob}[thm]{Problem}

\numberwithin{equation}{section}

\counterwithin{Question}{section}
\theoremstyle{remark}
\newtheorem{rem}[thm]{Remark}

\newcommand{\setone}{\mathbbm{1}}
\newcommand{\setA}{\mathbb{A}}
\newcommand{\setB}{\mathbbm{B}}
\newcommand{\setT}{{\mathbb{T}}}
\newcommand{\setQ}{\mathbb{Q}}
\newcommand{\setR}{{\mathbb{R}}}

\newcommand{\setZ}{{\mathbb{Z}}}
\newcommand{\setN}{{\mathbb{N}}}

\allowdisplaybreaks

\title{Weighted ergodic averages along subpolynomials in Hardy fields and applications}
\author{Vitaly Bergelson}
\address{Department of Mathematics, The Ohio State University, Columbus, OH, 43210, USA}
\email{vitaly@math.ohio-state.edu}

\author{Sovanlal Mondal}
\address{Department of Mathematics, Oregon State University, Corvallis, OR, 97331, USA}
\email{sovanlal.mondal01@gmail.com}

\author{Younghwan Son}
\address{Department of Mathematics, POSTECH, 77 Cheongam-Ro, Nam-Gu, Pohang, Gyeongbuk, Korea 37673 }
\email{yhson@postech.ac.kr}
\begin{document}
\begin{abstract}
 We establish new pointwise convergence results for weighted ergodic averages along sequences of the form
\(
(\lfloor a(n) \rfloor)_{n \in \mathbb{N}},
\)
where $a(x)$ is a subpolynomial function in a Hardy field. 
 For example, we establish pointwise convergence of logarithmic averages along sequences of the form $(\lfloor n^k + \log^{c} n \rfloor)_{n \in \mathbb{N}}$, where $k \in \mathbb{N} \cup \{0\}$ and $c > 0$. This result should be juxtaposed with the fact that either for $k=0$ or for $k \geq 2$ and for sufficiently small $c>0$ (depending on $k$), the standard ergodic averages along these sequences fail to converge pointwise.  
 
We also obtain pointwise joint ergodicity results for multiple weighted ergodic averages along slow Hardy field functions. For example, it follows from our results that 
for $c> 0$ and for any $f, g \in L^{\infty} (\lambda)$,
\begin{equation*}
\lim_{N \rightarrow \infty} \frac{1}{\log N } \sum_{n=1}^{N} \frac{1}{n} f(T_b^{\lfloor \log^c n \rfloor}x)  \, g(T_G^{\lfloor \log^c n \rfloor} x) = \int f \, d \lambda \cdot \int g \, d \mu_G
\quad \text{for almost every } x \in [0,1],
 \end{equation*}
where $T_b:[0,1] \rightarrow [0,1]$ is the times-$b$ map defined by $T_b x = bx \,  \bmod \, 1 $ and $T_G:[0,1] \rightarrow [0,1]$ is the Gauss map defined by $T_G(x) = \frac{1}{x} \bmod \, 1$ for $x \ne 0$ and $T_G (0) =0$.
Here $\lambda$ is the Lebesgue measure on $[0,1]$ and $\mu_G$ is the Gauss measure on $[0,1]$ given by $\mu_G (A) = \frac{1}{ \log 2} \int_A \frac{1}{1+x} dx$ for any measurable set $A \subset [0,1]$.
\end{abstract}
\maketitle
\date{\today}
\markright{WEIGHTED AVERAGES ALONG SUBPOLYNOMIALS}
\tableofcontents

\section{Introduction}\label{sec:intro}

A real sequence $(x_n)_{n \in \mathbb{N}}$ is said to be uniformly distributed $\bmod \, 1$ if for any continuous function $f$ on the interval $[0,1]$,
\begin{equation*} 
\lim_{N \rightarrow \infty} \frac{1}{N} \sum_{n=1}^N  f (\{x_n\}) = \int_0^1 f(x) \, dx.
\end{equation*}
Here, $\{ \cdot \}$ denotes the fractional part.
It is well known that the sequence $(a(\log n)^{1+\epsilon})_{n \in \mathbb{N}}$, where $a \in \mathbb{R}$ with $ a \ne 0 $, and $ \epsilon >0$, is uniformly distributed $\bmod \, 1$. 
This fact implies, via the spectral theorem, that for any measure preserving transformation $T$ on a probability space $(X, \mathcal{B}, \mu)$ and for any $f \in L^2(\mu)$, the averages  
\begin{equation}
\label{eq1.1:int:log}
 \frac{1}{ N} \sum_{n=1}^N f(T^{\lfloor (\log n)^{1+ \epsilon} \rfloor} x ) 
\end{equation}
converge in $L^2$-norm, and if $T$ is ergodic, the limit is equal to $\int f \, d \mu$. 
Somewhat surprisingly, even for $f \in L^{\infty} (\mu)$, the averages in \eqref{eq1.1:int:log} do not necessarily converge almost everywhere. As a matter of fact,  it is known that for any aperiodic measure preserving system $(X, \mathcal{B}, \mu, T)$, where $(X, \mathcal{B}, \mu)$ is a standard Lebesgue space (that is, it is measure theoretically isomorphic to $[0,1]$ with the Lebesgue measure), the averages \eqref{eq1.1:int:log} fail to converge almost everywhere for functions of the form $\setone_A$ for a generic Borel set $A$ (see \cite[Theorem 1.3]{JR} and \cite[Theorem 1.5]{LM}).
On the other hand, one can show that for any $f \in L^1(\mu)$, the logarithmic averages  
\begin{equation}
\label{eq:1.2:log-int}
 \frac{1}{ \log N} \sum_{n=1}^N \frac{1}{n} f(T^{\lfloor (\log n)^{1+ \epsilon} \rfloor} x ) 
\end{equation}
converge for almost every $x$ {(see Theorem \ref{thm:sufficient:int} below).}

The situation with the sequence $(\log n)_{n \in \mathbb{N}}$ is even more interesting. It is not hard to show that despite the fact that the sequence $( \{\log n \})_{n \in \mathbb{N}}$ is dense in the interval $[0,1]$, one can find a function $f \in C ([0,1])$ such that the averages
\[ \frac{1}{N} \sum_{n=1}^N f (\{\log n\})\]
do not converge as $N \rightarrow \infty$, and so $(\log n)_{n \in \mathbb{N}}$ is not uniformly distributed $\bmod \, 1$. (See Example 2.4 in \cite{KNbook}.) 
However, the sequence $(a \log n)_{n \in \mathbb{N}}$, where $a \in \mathbb{R}$ with $ a \ne 0 $, is uniformly distributed with respect to the method of the logarithmic averages which we already encountered in \eqref{eq:1.2:log-int}. 
Indeed one can show that for any $f \in C ([0,1])$,
\begin{equation} \label{eq:log:average:intro}
\lim_{N \rightarrow \infty} \frac{1}{\log N} \sum_{n=1}^N \frac{1}{n} f (\{a \log n\}) = \int_0^1 f(x) \, dx.
\end{equation}
(See \cite{Tsuji}.)

Let $T$ be a measure preserving transformation on a probability space $(X, \mathcal{B}, \mu)$.
Applying the spectral theorem, one obtains from \eqref{eq:log:average:intro} that for any $f \in L^2(\mu)$, the averages  
\begin{equation}
\label{eq1:intro:ex:mean1}
 \frac{1}{\log N} \sum_{n=1}^N \frac{1}{n} f(T^{\lfloor \log n \rfloor} x ) 
\end{equation}
converge in $L^2(\mu)$ and if $T$ is ergodic, the limit is equal to $\int f \, d \mu$.

In light of formula \eqref{eq1:intro:ex:mean1}, it is natural to inquire if  for $f \in L^1(\mu)$, the averages in \eqref{eq1:intro:ex:mean1}
converge almost everywhere. The answer to this question is positive; this fact is a rather special case of a more general theorem (Theorem \ref{thm:sufficient:int}), which deals with pointwise convergence along {\it regular} slowly growing functions.

In order to describe the results obtained in this paper, we have to introduce several pertinent notions.  
First of all, we briefly recall some basic facts about Hardy fields. 
(See \cite{BOS3} and some references therein for more information about Hardy fields.)

Let $\mathbf B$ be the set of germs at $+ \infty$ of continuous real functions on $\mathbb{R}$.
Note that $\mathbf B$ forms a ring with respect to pointwise addition and multiplication.
A {\em Hardy field} is any subfield of $\mathbf B$ that is closed under differentiation, and we denote by ${\mathbf U}$  the union of all Hardy fields.
A classical example of a Hardy field is the class $\mathbf{L}$ of logarithmico-exponential functions introduced by Hardy in \cite{Har1, Har2}, that is, the collection of all germs at $+ \infty$ of real-valued functions that can be constructed using the real constants, the functions $\exp x$ and $\log x$ and the operations of addition, multiplication, division and composition of functions. 

For any $u \in {\mathbf U}$, $\lim\limits_{x \rightarrow \infty} u(x)$ exists in $\mathbb{R} \cup \{- \infty, \infty\}$; thus periodic functions such as $\sin x$ and $\cos x$ do not belong to ${\mathbf U}$. 
If $u_1$ and $u_2$ belong to the same Hardy field, then the limit $\lim\limits_{x \rightarrow \infty} (u_1(x) - u_2(x))$ and the limit $\lim\limits_{x \rightarrow \infty} \frac{u_1(x)}{u_2(x)}$ exist (they may be infinite). 

 For two functions $u_1(x), u_2(x)$ with $u_2(x) \geq 0$, we write $u_1(x) \ll u_2(x)$ when there exists $C > 0$ such that $|u_1(x)| \leq C u_2(x)$ for all large enough $x$. Also we write $u_1(x) \prec u_2(x)$ when $\lim\limits_{x \rightarrow \infty} \frac{u_1(x)}{u_2(x)} = 0$.
A function $f \in \mathbf{U}$ is said to be subpolynomial if  $f(x) \ll x^n$ for some positive integer $n$.  

If a Hardy field is not a proper subfield of a larger Hardy field, then it is called a maximal Hardy field. 
Let $\mathbf{E}$ be the intersection of all maximal Hardy fields. 
It is known that $\mathbf{E}$ is closed under integration and composition, and contains the Hardy class $\mathbf{L}$. 
In fact, $\mathbf{E}$ is strictly larger than $\mathbf{L}$, as $\sin (1/x) \in \mathbf{E} \setminus \mathbf{L}$. Moreover, while $\mathbf{L}$ is not closed under integration, $\mathbf{E}$ is closed under integration. 
 (See Corollary 8.3 in \cite{Bos4} and Introduction in \cite{Bos1982}.)  

\begin{defn}
Let $a(x) \in \mathbf{U}$. Let $W(x) \in \mathbf{E}$ be such that $w(x):= W'(x)$ is positive and non-increasing and $ \lim_{x \rightarrow \infty} W(x) = \infty$. Given $p \geq 1$,
we will say that the sequence $\mathsf{a} := (\lfloor a(n) \rfloor)_{n \in \mathbb{N}}$ is {\em pointwise $L^p$-good for $W$-averages}, if for any invertible probability measure preserving system $(X, \mathcal{B}, \mu, T)$ and $f \in L^p(\mu)$, the averages 
\begin{equation}
\label{def:ergodic}
\setA_N^{\mathsf{a}, W} f(x) := \frac{1}{W(N)} \sum_{n=1}^N w(n) f(T^{\lfloor a(n) \rfloor} x)
\end{equation}
converge for almost every $x$.  

We say that the sequence $(\lfloor a(n) \rfloor)_{n \in \mathbb{N}}$ is {\em norm good for $W$-averages} if for any invertible probability measure preserving system $(X, \mathcal{B}, \mu, T)$ and for $f \in L^2(\mu)$, 
\begin{equation*} 
\lim_{N \rightarrow \infty} \setA_N^{\mathsf{a}, W}f(x)  \text{ converges in } L^2 (\mu)
\end{equation*}
and is {\em ergodic for $W$-averages} if for any invertible ergodic probability measure preserving system $(X, \mathcal{B}, \mu, T)$ and $f \in L^2(\mu)$,  
$$\lim_{N \rightarrow \infty} \setA_N^{\mathsf{a}, W}f(x) = \int f \, d \mu \text{ in } L^2 (\mu).$$
(Clearly, by the ergodic decomposition theorem, any sequence which is ergodic for $W$-averages is norm good for $W$-averages.)

We also say that the sequence $(\lfloor a(n) \rfloor)_{n \in \mathbb{N}}$ is {\em pointwise $L^p$-bad} for $W$-averages if for every aperiodic probability measure preserving system $(X, \mathcal{B}, \mu, T)$, there exists a function $f \in L^{p}(\mu)$ such that the corresponding weighted averages $\setA_N^{\mathsf{a}, W} f(x)$
fail to converge for almost every $x$. If $W(N) =N$ (so that $w(n) =1$), we simply say that $(\lfloor a(n) \rfloor)_{n \in \mathbb{N}}$ is pointwise $L^p$-good, norm good, ergodic, and pointwise $L^p$-bad respectively, omitting the phrase ``for $W$-averages". If $W(N) = \log N$ (respectively, $W(N) = \log \log N$), we refer to the corresponding averages as $\log$-averages (respectively, $\log\log$-averages).
\end{defn}

\begin{rem}
Note that in the formula \eqref{def:ergodic}, we use the sequence $W(N), N \in \mathbb{N},$ for the normalization, rather than the sequence of discrete sums $\sum_{n=1}^N w(n)$. 
In light of the fact
$ \lim\limits_{N \rightarrow \infty} \frac{\sum_{n=1}^N w(n)}{W(N)} = 1$, this choice will not affect any of our results.  
\end{rem}

The following theorem establishes the pointwise convergence of the weighted averages for a large class of slowly growing functions.
Note that the averages in \eqref{eq1:intro:ex:mean1} is a very special instance of applicability of this theorem. 
\begin{thm}[Theorem \ref{thm:sufficient}]
\label{thm:sufficient:int}
Let $a(x) \in \mathbf{U}$ satisfy $1 \prec a(x) \prec x$. 
Let $ W(x) \in \mathbf{E}$ be such that $w(x):= W'(x)$ is positive and non-increasing and $\lim_{x \rightarrow \infty} W(x) = \infty$.
Suppose that there exists $c > 0$ such that for all large enough $x$,  
\begin{equation*}
W(x) \leq |a(x)|^c. 
\end{equation*}
Then $(\lfloor a(n) \rfloor)_{n \in \mathbb{N}}$ is pointwise $L^1$-good and ergodic for $W$-averages. 
\end{thm}

In contrast to Theorem \ref{thm:sufficient:int}, the following theorem provides a condition when a sequence $(\lfloor a(n) \rfloor)_{n \in \mathbb{N}}$ is pointwise $L^{\infty}$-bad for $W$-averages. 
\begin{thm}[Theorem \ref{thm:necessary}]
\label{thm:necessary:int}
Let $a(x) \in \mathbf{U}$ such that $1 \prec a(x) \prec x$. 
Let $W(x) \in \mathbf{E}$ be such that $w(x):= W'(x)$ is positive and non-increasing and $\lim_{x \rightarrow \infty} W(x) = \infty$. 
Suppose that there exists $c>0$ such that 
\begin{equation*}
W(x) \geq \exp (|a(x)|^{c})
\end{equation*}
for all sufficiently large $x$.
Then $(\lfloor a(n) \rfloor)_{n \in \mathbb{N}}$ is pointwise $L^{\infty}$-bad for $W$-averages. 
\end{thm}

Before providing some illustrative examples, we formulate a special case of Theorem 4.10 in \cite{BKS2} which provides a characterization of {\em norm} convergence along $(\lfloor a(n) \rfloor)_{n \in \mathbb{N}}$ when $a(x) \in \mathbf{U}$ with $1 \prec a(x) \prec x$. 
\begin{thm}[cf. Theorem 4.10 in \cite{BKS2}]
\label{thm4.10:BKS2}
Let $a(x) \in \mathbf{U}$ such that $1 \prec a(x) \prec x$. 
Let $W(x) \in \mathbf{E}$ be such that $w(x):= W'(x)$ is positive and non-increasing and $\lim\limits_{x \rightarrow \infty} W(x) = \infty$. 
Then the sequence $(\lfloor a(n) \rfloor)_{n \in \mathbb{N}}$ is norm-good for $W$-averages if and only if 
\begin{equation*}
\log W(x) \prec  a(x).
\end{equation*}
Moreover, if  $(\lfloor a(n) \rfloor)_{n \in \mathbb{N}}$ is norm-good for $W$-averages, then it is ergodic for $W$-averages.
\end{thm}

We list now some examples which present ``positive" and ``negative" results pertaining to pointwise and norm convergence which can be derived from Theorems \ref{thm:sufficient:int}, \ref{thm:necessary:int} and \ref{thm4.10:BKS2}.
\begin{exmp} \label{ex1:sec1}
\begin{enumerate}
\item Let $c>0$. The sequence $( \lfloor (\log n)^c \rfloor)_{n \in \mathbb{N}}$ is pointwise $L^{\infty}$-bad, but pointwise $L^1$-good for $\log$-averages. On the other hand, it is norm good if and only if $c>1$, while for $0<c\leq 1$ it is norm good for $\log$-averages.
\item The sequence $( \lfloor \log n \log\log n \rfloor)_{n \geq 2}$ is pointwise $L^{\infty}$-bad, but pointwise $L^1$-good for $\log$-averages. Moreover, it is norm good.
\item The sequence $( \lfloor (\log\log n)^2 \rfloor)_{n \geq 2}$ is pointwise $L^{\infty}$-bad for $\log$-averages, but pointwise $L^1$-good for $\log\log$-averages. Moreover, it is norm good for $\log$-averages.
\end{enumerate}
\end{exmp}

In light of Theorems~\ref{thm:sufficient:int} and \ref{thm:necessary:int}, it is natural to consider the following problem.
\begin{prob}\label{problem1}
Let $W(x) \in \mathbf{E}$ be such that $w(x):= W'(x)$ is positive and non-increasing and $\lim\limits_{x \rightarrow \infty} W(x) = \infty$.
 Suppose that $a(x) \in \mathbf{U}$ with $1 \prec a(x) \prec x$ and assume that for any $c_1, c_2 > 0$,
 \begin{equation}
 \label{eq:prob1.7:intro}
 |a(x)|^{c_1} \leq  W(x) \leq \exp (|a(x)|^{c_2}) 
 \end{equation}
 for all large enough $x$.
What can be said about pointwise convergence/divergence of $\setA_N^{\mathsf{a}, W} f(x)$, when $a(x)$ and $W(x)$ satisfy \eqref{eq:prob1.7:intro}?

Consider, for example, $a(x) = \log x$ and $W(x) = \exp \left((\log \log x)^2\right)$.
Clearly, formula \eqref{eq:prob1.7:intro} is satisfied. 
In this case, by Theorem 4.10 in \cite{BKS2}, for every ergodic measure preserving system $(X, \mathcal{B}, \mu, T)$ and every $f \in L^2(\mu)$, we have 
\begin{equation} 
\label{eq:prob1.7:conv:intro}
\lim_{N \rightarrow \infty} \setA_N^{\mathsf{a}, W} f(x) = \int f \,  d\mu \text{ in } L^2(\mu).
\end{equation}
But it is unknown if one has almost everywhere convergence in \eqref{eq:prob1.7:conv:intro}.
\end{prob}

Next, we consider the averages \eqref{def:ergodic} for general subpolynomial functions $a(x) \in \mathbf{U}$. 
In the formulations of theorems below, we will be invoking the {\em canonical representation} of a subpolynomial function which states that a subpolynomial function $a(x)\in \mathbf{U}$ can be written as  
\begin{equation*} \label{eq:canonical-decomposition:def}
a(x) = p(x)+r(x),
\end{equation*}
where $p(x)$ is a polynomial and $r(x)$ is a non-polynomial\footnote{ A function $r(x)$ in a Hardy field is called {\em non-polynomial} if it is subpolynomial and for all $k \in \mathbb{N} \cup \{0\}$, $r(x) \prec x^k$ or $x^k \prec r(x)$.} and satisfies $r(x)\prec x^s$ for each non-zero term $c_s x^s$ of $p(x)$. (See  Lemma 3.1 in \cite{BKQW}.)
For example, if $a(x)=x^3+x^2+x^{3/2}+x$, then $p(x)=x^3+x^2$ and $r(x)=x^{3/2}+x$, whereas if $a(x)=x^{5/2}+ x^2 +x$, then $p(x)=0$ and $r(x)=x^{5/2}+x^2 + x$. 

We first state the following general sufficient condition for pointwise $L^2$-goodness.
\begin{thm}
\label{sparse2:intro}
Let $a(x) \in \mathbf{U}$ be a subpolynomial function with the canonical decomposition $a(x) = p(x) + r(x)$.
Suppose that $r(x) \succ x^{\epsilon}$ for some $\epsilon > 0$. Then $(\lfloor a(n) \rfloor)_{n \in \mathbb{N}}$ is pointwise $L^2$-good.
\end{thm}
\begin{rem}
By Theorem 4.10 in \cite{BKS2}, the sequence $(\lfloor a(n) \rfloor)_{n \in \mathbb{N}}$ in Theorem \ref{sparse2:intro} is ergodic.
\end{rem}

Theorem \ref{sparse2:intro} follows by considering separately three different cases according to the degree of the polynomial part $p(x)$.
First, when $\deg(p) \geq 2$, Theorem \ref{sparse2:intro} is a special case of the following result.
\begin{thm}[Theorem 3.5 in \cite{BKQW}]
\label{sparse:intro:prev}
Let $a(x) \in \mathbf{U}$ be a subpolynomial function with the canonical decomposition $a(x) = p(x) + r(x)$.
Set $d := \deg (p)$. Suppose that $d \geq 2$ and $r(x) \succ (\log x)^{2^{d+1} - 1 + \epsilon}$ for some $\epsilon > 0$. Then $(\lfloor a(n) \rfloor)_{n \in \mathbb{N}}$ is pointwise $L^2$-good.
\end{thm}

It is instructive to compare the above ``positive" result in Theorem \ref{sparse:intro:prev}  with the following ``negative" results, which were also proved in \cite{BKQW}. 
\begin{thm}
\label{thm:sparse:epsilon:int}
Let $p(x)$ be a polynomial in $ C \setQ (x)$\footnote{$ C \setQ (x)$ denotes the set of all real constant multiples of polynomials in $\mathbb{Q}[x]$.}. 
\begin{enumerate}
\item (Theorem 3.6 of \cite{BKQW}) If $\deg(p) \geq 2$ and $1 \prec r(x) \prec \log x \exp ((\log \log x)^m)$ for some $0 \leq m < 1$, then $(\lfloor p(n)+ r(n)\rfloor)$ is pointwise $L^2$-bad.
\item (Theorem 3.7 of \cite{BKQW})  Suppose that the lowest order term of $p(x)$ is of degree $k\geq 3$. If $1\prec r(x) \prec (\log x)^m $, where $m< 2k/(k+2)$, then $(\lfloor p(n)+ r(n)\rfloor)$ is pointwise $L^2$-bad.
\end{enumerate}
\end{thm}

Next, in the case $\deg(p)=1$, Theorem \ref{sparse2:intro} follows from the following result, which is proved in Section \ref{deg:1}.
\begin{thm}[Theorem \ref{thm:deg1}]\label{thm:deg1:int}
Let $a(x)\in \mathbf{U}$  have the canonical decomposition $a(x)=cx+r(x)$ with $c\neq 0$.
If either, $r(x) \succ \log^ {2+\epsilon} x$ for some $\epsilon>0$, or $c=\frac{1}{m}$ for some $m\in \setZ$,  then $(\lfloor a(n) \rfloor)_{n\in \setN}$ is pointwise $L^2$-good.
\end{thm}

\begin{rem}
In view of the condition $r(x) \succ (\log x)^{2^{d+1} - 1 + \epsilon}$ in Theorem \ref{sparse:intro:prev} for $d = \deg (p) \geq 2$, one might expect that the corresponding condition for $\deg (p) =1$ in Theorem \ref{thm:deg1:int} would be $r(x) \succ \log^{3+\epsilon} x$. However, Theorem \ref{thm:deg1:int} shows that the weaker condition $r(x) \succ \log^ {2+\epsilon} x$ is already sufficient for the sequence $(\lfloor a(n) \rfloor)_{n\in \setN}$ to be pointwise $L^2$-good.
\end{rem}

\begin{prob}
Let $a(x) \in \mathbf{U}$ have the canonical decomposition $a(x) = cx + r(x)$, where $c \ne 0$ and $1/c \notin \mathbb{Z}$.
If $1\prec |r(x)| \ll \log  x$ then by \cite[Theorem 3.2]{BKQW}, the sequence $(\lfloor a(n) \rfloor)_{n \in \mathbb{N}}$ is bad for mean convergence, and hence it is pointwise $L^2$-bad.

This observation, together with Theorem \ref{thm:deg1:int}, leads to the following question.
What can be said about the pointwise convergence or divergence of 
$$\setA_N^{\mathsf{a}} f(x) = \frac{1}{N} \sum_{n=1}^N f(T^{\lfloor a(n) \rfloor} x), \text{ when } \log x \prec |r(x)| \ll \log^{2+\epsilon} x \text{ for every } \epsilon >0?$$
\end{prob}

Finally, when $\deg(p) =0$, Theorem \ref{sparse2:intro} is a consequence of the next result. 
\begin{thm}[Theorem B in \cite{BW1996}]
Let $a(x) \in \mathbf{U}$ be a subpolynomial. 
\begin{enumerate}
\item If $x^{k-1} \prec a(x) \prec x^k$ for some $k \geq 2$, then $(\lfloor a(n) \rfloor)_{n \in \mathbb{N}}$ is pointwise $L^p$-good for any $p>1$.
\item If $x^{\delta} \prec a(x) \prec x$ for some $\delta >0$, then $(\lfloor a(n) \rfloor)_{n \in \mathbb{N}}$ is pointwise $L^1$-good.  
\end{enumerate}
\end{thm}

\begin{rem}
In light of Theorems \ref{sparse:intro:prev} and \ref{thm:deg1:int}, one may expect that if $\deg (p) =0$, there exists $c>0$ such that if $r(x) \succ (\log x)^{c + \epsilon}$ for some $\epsilon > 0$, then $(\lfloor a(n) \rfloor)_{n \in \mathbb{N}}$ is pointwise $L^2$-good. 
However, by Example \ref{ex1:sec1} (1) which is based on Theorem \ref{thm:necessary:int}, this is not the case. 
\end{rem}

In Sections \ref{sec:sparse1} and \ref{sec:sparse2}, we study weighted averages for subpolynomial functions and, in particular, derive the following result for logarithmic averages version of Theorem \ref{thm:sparse:epsilon:int}. 
\begin{thm} [cf. Theorem \ref{thm:sparse2'}] 
\label{sparse2':intro}
Let $a(x) \in \mathbf{U}$ be a subpolynomial function with canonical decomposition $a(x) = p(x) + r(x)$. Suppose that $|r(x)| \gg (\log x)^{\eta}$ for some $\eta > 1$. Then $(\lfloor a(n) \rfloor)_{n \in \mathbb{N}}$ is pointwise $L^2$-good and ergodic for $\log$-averages.  
\end{thm}

Question 10.1 in \cite{BKQW} asks whether the sequence $(\lfloor n^2 + \log^2 n \rfloor)_{n \in \mathbb{N}}$ is pointwise $L^2$-good. To the best of our knowledge, this question is still open. 
More generally, it is natural to inquire about the almost everywhere convergence for the sequences $(\lfloor n^k + \log^c n \rfloor)_{n \in \mathbb{N}}$, where $k \in \mathbb{N} \cup \{0\}$ and $c > 0$.
Theorems formulated above throw some light on this question.  
The following example provides a summary of the current state of the art.
\begin{exmp} \label{ex-new:sec1}
Let $a(x) = x^k + \log^c x$, where $k \in \mathbb{N} \cup \{0\}$ and $c > 0$.
\begin{itemize}
\item If $k=0$, then $(\lfloor a(n) \rfloor)_{n \in \mathbb{N}}$ is pointwise $L^{\infty}$-bad, but it is pointwise $L^1$-good for log-averages (Theorems \ref{thm:sufficient:int} and \ref{thm:necessary:int}). (See also Example \ref{ex1:sec1}.)
\item If $k=1$, then $(\lfloor a(n) \rfloor)_{n \in \mathbb{N}}$ is pointwise $L^2$-good (Theorem \ref{thm:deg1:int}).
\item If $k = 2$, then $(\lfloor a(n) \rfloor)_{n \in \mathbb{N}}$ is  pointwise $L^2$-good if $c > 7$ (Theorem 3.5 of \cite{BKQW}) and pointwise $L^2$-bad if $0 < c < 2$ (Theorem 3.6 of \cite{BKQW}). It is pointwise $L^2$-good for log-averages if $c >1$ (Theorem \ref{sparse2':intro}).
\item If $k \geq 3$, then $(\lfloor a(n) \rfloor)_{n \in \mathbb{N}}$ is  pointwise $L^2$-good if $c> 2^{k+1} -1$ (Theorem 3.5 of \cite{BKQW}), and pointwise $L^2$-bad if $0< c < \frac{2k}{k+2}$ (Theorem 3.7 of \cite{BKQW}).  Moreover, if $c>1$, then it is pointwise $L^2$-good for log-averages (Theorem \ref{sparse2':intro}). 
\end{itemize}
\end{exmp}

It is worth mentioning that the only previously known example which demonstrates a phenomenon similar to that exhibited for ``sparse" sequences by Theorems 3.6, 3.7 of \cite{BKQW} and Theorem \ref{sparse2':intro} (pointwise $L^2$-bad for conventional Ces\`aro averages but pointwise $L^2$-good for weighted averages) involves arithmetic functions such as $\Omega(n)$ (the total number of prime divisors of a positive integer $n$ including their multiplicities), which has oscillating behavior (see, for example, Theorems 1.1 - 1.3 in \cite{LM} and Theorem 1.7 in \cite{BRR}).

Example \ref{ex-new:sec1} leads to the following open questions. 
\begin{ques}
\begin{enumerate}
\item For what values of $c\in [2,7]$, is $(\lfloor n^2 + \log^{c} n \rfloor)_{n \in \mathbb{N}}$ pointwise $L^2$-good?
\item For what values of $c \in \left[\frac{2k}{k+2}, 2^{k+1}-1 \right]$, is $(\lfloor n^k + \log^{c} n \rfloor)_{n \in \mathbb{N}}$ pointwise $L^2$-good? 
\item If $k \geq 2$ and $0< c \leq 1$, then is $ \left(\lfloor n^k+\log^c n\rfloor\right)$ pointwise $L^2$-good for $\log$-averages? or, say, for $\log \log$-averages?
\end{enumerate}
\end{ques}

In Section \ref{sec:JE}, we will utilize almost everywhere convergence results formulated above to obtain new theorems concerning pointwise joint ergodicity in weighted setting.
Let $T_1, \dots, T_k$ be measurable transformations on a probability space $(X, \mathcal{B}, \mu)$ such that for each $i=1, \dots, k$, there exists $T_i$-invariant probability measure $\mu_i$, which is equivalent to $\mu$.  
Recall that $T_1, \dots, T_k$ are said to be {\em pointwise jointly ergodic} if for any bounded measurable functions $f_1, \dots, f_k$ on $X$, 
\[ \lim_{N \rightarrow \infty} \frac{1}{N} \sum_{n=1}^N \prod_{i=1}^k f_i (T_i^n x) = \prod_{i=1}^k \int f_i \, d \mu_i \quad \text{for a.e. } x.\]

The following proposition is an extension of Theorem \ref{thm:sufficient:int} to pointwise joint ergodicity when the invovled functions belong to $L^{\infty}$.
\begin{prop}[Proposition \ref{thm:je1}]
\label{thm:je1:int}
Let $a(t) \in \mathbf{U}$ such that $1 \prec a(t) \prec t$. 
Let $W(t) \in \mathbf{E}$ be such that $w(t):= W'(t)$ is positive and non-increasing and $\lim\limits_{t \rightarrow \infty} W(t) = \infty$.
Suppose that there exists $c > 0$ such that for all large enough $t$,  
\begin{equation*}
W(t) \leq |a(t)|^c. 
\end{equation*} 
If $T_1, \dots, T_k$ are pointwise jointly ergodic, then for any bounded measurable functions $f_1, \dots, f_k$ on $X$, 
\begin{equation*}
\label{eq:thm:je1}
\lim_{N \rightarrow \infty} \frac{1}{W(N)} \sum_{n=1}^N w(n) \prod_{i=1}^k f_i (T_i^{\lfloor a(n) \rfloor} x) = \prod_{i=1}^k \int f_i \, d \mu_i \quad \text{for a.e. } x.
\end{equation*}
\end{prop}

In  \cite{BerSon2025}, the phenomenon of joint ergodicity of piecewise monotone maps on the unit interval $[0,1]$ was investigated. 
In particular, the paper establishes the pointwise joint ergodicity of the following piecewise monotone maps.
Let $T_{b} x = bx \, \bmod \, 1$, ($b \in \mathbb{N}$, $b \geq 2$), $T_{\beta} x = \beta x \, \bmod \, 1$, ($\beta > 1$, $\beta \notin \mathbb{N}$), and let $T_G$ be the Gauss map defined by $T_G(x) = \frac{1}{x} \, \bmod \,1$. 
It is well-known that $T_{b}$ is $\lambda$-preserving, where $\lambda$ denotes the Lebesgue measure  on the unit interval $[0,1]$, and $T_G$ is $\mu_G$-preserving, where $\mu_G(A) = \frac{1}{\log 2} \int_A \frac{dx}{1+x}$. 
Moreover, by a theorem of R{\'e}nyi (\cite{Ren}), there exists a unique $T_{\beta}$-ergodic probability measure $\mu_{\beta}$, which satisfies $1 - 1/\beta \leq \frac{d \mu_{\beta}}{d \lambda} \leq \frac{1}{1- 1/\beta}$. 
It was shown in \cite{BerSon2025} that $T_b, T_{\beta}$ and $T_G$ are pointwise jointly ergodic if $\log \beta \ne \frac{\pi^2}{ 6 \log 2}$. 

Combining Proposition \ref{thm:je1:int} with the pointwise joint ergodicity result from \cite{BerSon2025}, we obtain the following theorem.
\begin{thm}
Let $a(t) \in \mathbf{U}$ such that $1 \prec a(t) \prec t$. 
Let $W(t) \in \mathbf{E}$ be such that $w(t):= W'(t)$ is positive and non-increasing and $\lim_{t \rightarrow \infty} W(t) = \infty$.
Suppose that there exists $c > 0$ such that for all large enough $t$,  
\begin{equation*}
W(t) \leq |a(t)|^c. 
\end{equation*} 
If $\log \beta \ne \frac{\pi^2}{ 6 \log 2}$, then for any bounded measurable functions $f_1, f_2, f_3$ on $[0,1]$, 
\[  \lim_{N \rightarrow \infty} \frac{1}{W(N)} \sum_{n=1}^N w(n) f_1 (T_b^{\lfloor a(n) \rfloor} x) f_2 (T_G^{\lfloor a(n) \rfloor} x) f_3 (T_{\beta}^{\lfloor a(n) \rfloor} x) =  \int f_1 \, d \lambda \int f_2 \, d \mu_G \int f_3 \, d \mu_{\beta}\] 
for almost every  $x$.
\end{thm}

We next turn to another class of results concerning pointwise joint ergodicity, which will be established in Section \ref{sec:JE}. These results concern rotations on the torus along sequences arising from Hardy fields.
For a finite set of functions $\mathcal{F} = \{u_1, \dots, u_k\}$, we define 
\[\text{span}_{\mathbb{Z}}^{*} \, \mathcal{F} = \left\{ \sum_{i=1}^k h_i u_i: (h_1, \dots, h_k) \in \mathbb{Z}^k \setminus \{(0, \dots, 0)\} \right\}.\]

\begin{thm}[Theorem \ref{thm:multiconv}]
\label{thm:multiconv:int}
Let $W(t) \in \mathbf{E}$ be such that $w(t):= W'(t)$ is positive and non-increasing and $\lim\limits_{t \rightarrow \infty} W(t) = \infty$.
Let $(p_n)_{n \in \mathbb{N}}$ be the sequence of prime numbers.
Let $a_1(t), \dots, a_l(t), b_1(t), \dots, b_m(t) \in \mathbf{U}$, and let $\alpha_1, \dots, \alpha_l, \beta_1, \dots, \beta_m \in \mathbb{R}$ be irrational numbers.
Define
$$\mathcal{F} = \{ c_i a_i(t): c_i = 1 \text{ or } \alpha_i, \, 1 \leq i \leq l\} \cup \{ d_j b_j(t \log t): d_j = 1 \text{ or } \beta_j, \, 1 \leq j \leq m \}.$$
Assume that 
\begin{enumerate}[(i)]
\item for all $i=1, 2, \dots, l$, $1 \prec a_i(t) \prec t$, 
\item for all $j=1, 2, \dots, m$, $1 \prec b_j(t) \prec \log t$, 
\item $a_1(t), \dots, a_l(t), b_1(t \log t), \dots, b_m(t \log t) \in \mathbf{H}$ for some Hardy field $\mathbf{H}$ and there exists $c > 0$ such that for all sufficiently large $t$,  
\begin{equation*}
W(t) \leq |a_i(t)|^c \quad \text{and} \quad W(t) \leq |b_j(t)|^c \text{ for any } i, j.
\end{equation*} 
\item for any $u(t) \in \text{span}_{\mathbb{Z}}^* \, \mathcal{F}$,
\begin{equation*}
\lim_{t \rightarrow \infty}  \frac{|u(t)|}{\log W(t)} = \infty.  
\end{equation*}
\end{enumerate}
Then, for any bounded measurable functions $f_1, \dots, f_l, g_1, \dots, g_m$ on $\mathbb{T} \, (:= \mathbb{R} / \mathbb{Z})$, 
\begin{equation*}
\label{eq3.4}
\lim_{N \rightarrow \infty} \frac{1}{W(N)} \sum_{n=1}^N w(n) \prod_{i=1}^l f_i (x + \lfloor a_i(n) \rfloor \alpha_i) \prod_{j=1}^m g_j( x+ \lfloor b_j(p_n) \rfloor \beta_j) = \prod_{i=1}^l \int f_i \, d\lambda \prod_{j=1}^m \int g_j \, d \lambda
\end{equation*}  
for $\lambda$-almost every $x$.
\end{thm}

By applying Theorem \ref{sparse2:intro} and Theorem \ref{sparse2':intro}, we obtain the following result.
\begin{thm}[Theorem \ref{thm:multiconv2}]
\label{thm:multiconv2:int}
Let $a_1(t), \dots, a_l(t)$ be subpolynomial functions in a Hardy field $\mathbf{H}$. 
Let $\alpha_1, \dots, \alpha_l$ be irrational numbers.
Define 
$$\mathcal{F} = \{ a_1(t), \dots, a_l(t),  \alpha_1 a_1(t), \dots, \alpha_l a_l(t) \}.$$
For $i= 1, 2 \dots, l$, let $a_i(t) = p_i(t) +r_i(t)$ be the canonical decomposition. 
\begin{enumerate}[(1)]
\item Assume that for some $\eta > 0$, $|r_i(t)| \gg t^{\eta}$ for every $i$, and for any $u(t) \in \text{span}_{\mathbb{Z}}^* \, \mathcal{F}$,
\begin{equation*}
\label{eq:ud:je:last}
\lim_{t \rightarrow \infty}  \frac{|u(t) - q(t)|}{\log t} = \infty \quad \text{for any } q(t) \in \mathbb{Q}[t]. 
\end{equation*}
Then, for any bounded measurable functions $f_1, \dots, f_l$ on $\mathbb{T}$, 
\begin{equation*}
\lim_{N \rightarrow \infty} \frac{1}{N} \sum_{n=1}^N \prod_{i=1}^l f_i (x + \lfloor a_i(n) \rfloor \alpha_i) = \prod_{i=1}^l \int f_i \, d \lambda
\end{equation*}  
for $\lambda$-almost every $x$.
\item Assume that for some $\eta > 1$, $|r_i(t)| \gg (\log t)^{\eta}$ for all $i$ and  for any $u(t) \in \text{span}_{\mathbb{Z}}^* \, \mathcal{F} $,
\begin{equation*}
\lim_{t \rightarrow \infty}  \frac{|u(t) -q(t)|}{ \log \log t} = \infty \quad \text{for any } q(t) \in \mathbb{Q}[t]. 
\end{equation*}
Then, for any bounded measurable functions $f_1, \dots, f_l$ on $\mathbb{T}$, 
\begin{equation*}
\lim_{N \rightarrow \infty} \frac{1}{\log N} \sum_{n=1}^N \frac{1}{n} \prod_{i=1}^l f_i (x + \lfloor a_i(n) \rfloor \alpha_i) = \prod_{i=1}^l \int f_i \, d \lambda 
\end{equation*}  
for $\lambda$-almost every $x$.
\end{enumerate}
\end{thm}

We conclude this introduction by presenting several consequences of Theorems \ref{thm:multiconv:int} and \ref{thm:multiconv2:int}. 
\begin{enumerate}
\item 
\begin{enumerate}
\item For any irrational numbers $\alpha$ and  $\beta$, and for any bounded $1$-periodic measurable functions $f$ and $g$,
\begin{equation*}
\lim_{N \rightarrow \infty} \frac{1}{\log \log  N} \sum_{n=2}^N \frac{1}{n \log n} f (x + \lfloor \log n \rfloor \alpha) \cdot g (x + \lfloor \log p_n \rfloor \beta) = \int f(t) \, dt  \int g(t) \, dt 
\end{equation*}  
for almost every $x$.
\item For any irrational numbers $\beta_1, \dots, \beta_m$, and any bounded $1$-periodic measurable functions $g_1, \dots, g_m$,
\begin{equation*}
\lim_{N \rightarrow \infty} \frac{1}{\log^{(m+1)} N} \sum_{n=A_m}^N \frac{1}{n \log n \log^{(2)} n \cdots \log^{(m)} n} \prod_{j=1}^m g_j (x + \lfloor \log^{(j)} p_n \rfloor \beta_j)  = \prod_{j=1}^m \int g_j(t) \, dt 
\end{equation*}  
for almost every $x$. (Here, $\log^{(k)}$ denotes the $k$-fold iteration of the logarithm and $A_m$ is the smallest integer such that $\log^{(m)} y > 0 $ if $y \geq A_m$.)
\end{enumerate}
\item 
\begin{enumerate}
\item Let $c_1, \dots, c_l$ be distinct positive real numbers.  For any irrational numbers $\alpha_1, \dots, \alpha_l$ and for any bounded measurable functions $f_1, \dots, f_l$,
\begin{equation*}
\lim_{N \rightarrow \infty} \frac{1}{N} \sum_{n=1}^N \prod_{i=1}^l f_i (x + \lfloor n^{c_i} \rfloor \alpha_i)  = \prod_{i=1}^l \int f_i(t) \, dt 
\end{equation*}  
for almost every $x$.
\item If $\alpha_1, \alpha_2$ are irrational numbers such that $1, \alpha_1, \alpha_2$ are rationally independent, then for any bounded measurable functions $f_1, f_2$,
\begin{equation*}
\lim_{N \rightarrow \infty} \frac{1}{ N} \sum_{n=1}^N f_1 (x + \lfloor \sqrt{n} \rfloor \alpha_1) \, f_2 (x + \lfloor \sqrt{n} + \log n \rfloor \alpha_2)   = \prod_{i=1}^2 \int f_i(t) \, dt 
\end{equation*}   
On the other hand, if $1, \alpha_1, \alpha_2$ are rationally dependent, then the corresponding Ces\`aro  averages may fail to converge almost everywhere. 
Indeed, if $\alpha_1 = \alpha_2$ and $f_1(x) = e(-x), f_2(x) = e(x)$, then these averages fail to converge for every $x$.
Nevertheless, for any irrational $\alpha_1, \alpha_2$ and for any bounded measurable functions $f_1, f_2$,
\begin{equation*}
\lim_{N \rightarrow \infty} \frac{1}{\log N} \sum_{n=1}^N \frac{1}{n} f_1 (x + \lfloor \sqrt{n} \rfloor \alpha_1) \, f_2 (x + \lfloor \sqrt{n} + \log n \rfloor \alpha_2)   = \prod_{i=1}^2 \int f_i(t) \, dt 
\end{equation*}  
for almost every $x$. 
\end{enumerate}
\end{enumerate}

The structure of the paper is as follows.
In Sections 2-5, we study weighted ergodic averages along sequences arising from Hardy fields.
We first consider slowly growing functions in Section 2, then sparse sequences in Section 3, and finally functions of degree one in Section 4.
Section 5 is devoted to the proof of Theorem \ref{sparse2':intro} on almost everywhere convergence of logarithmic averages.
Finally, Section 6 presents pointwise joint ergodicity in the weighted setting as an application of the preceding results.

We will use the following notation throughout the paper.
\begin{itemize}
\item $\mathbf{U}$ denotes the union of all Hardy fields.
\item $\mathbf{E}$ denotes the intersection of all maximal Hardy fields.
\item For $k\in \setN$, $\lambda^k$ denotes the Lebesgue measure on $\setR^k$ or $\setT^k$.
\item For two functions $f$ and $g$ with $g\geq 0$, $f\ll g$ means there exist constants $c>0$ and $M > 0$ such that $|f(x)|\leq cg(x)$ for all $x\geq M$. If $f$ is also non-negative, then $g\gg f$ means $f\ll g$.
\item For two functions $f$ and $g$, $f\prec g$ means $\displaystyle\lim_{x\to \infty} \dfrac{f(x)}{g(x)}=0$, and $f\succ g$ means $\displaystyle\lim_{x\to \infty} \dfrac{f(x)}{g(x)}=\pm \infty.$
\item For a real-valued sequence $(a(n))$, a measure preserving system $(X, \mathcal{B}, \mu, T)$ and $f \in L^p(\mu)$, $\setA_N^{\mathsf{a},W}f(x)$ denotes the following:  \begin{equation*}
\setA_N^{\mathsf{a}, W} f(x) := \frac{1}{W(N)} \sum_{n=1}^N w(n) f(T^{\lfloor a(n) \rfloor} x).
\end{equation*}
\item For $t \in \mathbb{R}$, $e(t) := e^{2 \pi i t}$.\\
\item For a positive integer $N$, $[N]$ will denote the set $\{1,2,\cdots,N\}$.
\end{itemize}

\section{Slow growing sequences}
\label{sec:slow}

In this section, we prove \Cref{thm:sufficient:int} and \Cref{thm:necessary:int}.
We consider weighted ergodic averages along slowly growing sequences.
Throughout this section, we assume that
\begin{equation}\label{eq:1'}
a(x)\in\mathbf U,\qquad 1\prec a(x)\prec x.
\end{equation}
\begin{rem}\label{rem:range}
Observe that since $a(x)$ satisfies \eqref{eq:1'}, the set $\{\lfloor a(n) \rfloor: n\in \setN\}$ contains all but finitely many positive (respectively, negative respectively) integers if $a(x)$ is eventually positive (respectively negative).
\end{rem}

We begin by recalling Theorem~\ref{thm:sufficient:int}.
\begin{thm}[Theorem \ref{thm:sufficient:int}]
\label{thm:sufficient}
Let $a(x) \in \mathbf{U}$ satisfy $1 \prec a(x) \prec x$. 
Let $ W(x) \in \mathbf{E}$ be such that $w(x):= W'(x)$ is positive and non-increasing and $\lim_{x \rightarrow \infty} W(x) = \infty$.
Suppose that there exists $c > 0$ such that for all sufficiently large $x$,  
\begin{equation*}
W(x) \leq |a(x)|^c. 
\end{equation*}
Then the sequence $(\lfloor a(n) \rfloor)_{n \in \mathbb{N}}$ is pointwise $L^1$-good and ergodic for $W$-averages: 
for any ergodic probability measure preserving system $(X, \mathcal{B}, \mu, T)$ and for every $f \in L^1$, the corresponding weighted averages $\setA_N^{\mathsf{a}, W} f(x)$ converge to $\int f$ a.e.
\end{thm}

The proof of \Cref{thm:sufficient} relies on the following auxiliary lemma
(see Lemma 7.1 in Chapter 1 of \cite{KNbook}).
\begin{lem}
\label{lem:KN:comparison}
    Let $w_1(n)$ and $w_2(n)$ be positive sequence such that for $i=1, 2$, $\lim_{n \rightarrow \infty} W_i(n) = \infty$, where $W_i (n) = \sum_{k=1}^n w_i(k)$.
    Suppose that either 
    \[ \frac{w_2(n+1)}{w_2 (n)} \leq \frac{w_1(n+1)}{w_1(n)} \,\, \text{ for all } n \geq 1\]
    or
    \[ \frac{w_2(n+1)}{w_2 (n)} \geq \frac{w_1(n+1)}{w_1(n)} \, \, \text{ and } \,\, \frac{W_1(n)}{w_1 (n)} \leq H \frac{W_2 (n)}{w_2(n)} \,\, \text{ for some } H \text{ and } n \geq 1.  \]
    Then, for any sequence $(a_n)_{n \in \mathbb{N}}$, if 
    \[ \lim_{N \rightarrow \infty} \frac{1}{W_1(N)} \sum_{n=1}^N w_1(n) a_n =a,\]
    then 
        \[ \lim_{N \rightarrow \infty} \frac{1}{W_2(N)} \sum_{n=1}^N w_2(n) a_n =a.\]
\end{lem}

For later use, we record the following two consequences of Lemma \ref{lem:KN:comparison}.
Taking $w_1(n) =1$ in Lemma \ref{lem:KN:comparison}, we obtain the following lemma.
\begin{lem}\label{lem:HW-2}
Let $(a_n)_{n \in \mathbb{N}}$ and $(d_n)_{n \in \mathbb{N}}$ be two sequences of real numbers such that $d_n>0$ for all $n\in\setN$. 
Set $D_N=\sum_{n = 1}^{N} d_n$. Let us assume that $\lim\limits_{N\to \infty}D_N = \infty$ and  $\lim\limits_{N \rightarrow \infty} \frac{1}{N} \sum\limits_{n = 1}^{N} a_n =a$. 
Then, the following holds.
\begin{enumerate}[(a)]
    \item If $(d_n)$ is non-increasing, then 
    $\lim\limits_{N \rightarrow \infty} \displaystyle\frac{1}{D_N}\sum_{n = 1}^{N} d_n a_n = a.$\\
    \item If $(d_n)$ is non-decreasing and 
    \begin{equation*}
    \sup_N \dfrac{N d_N}{D_N}<\infty,
    \end{equation*}
    then $\lim\limits_{N \rightarrow \infty} \displaystyle\frac{1}{D_N}\sum_{n = 1}^{N} d_n a_n = a$.
\end{enumerate}
\end{lem}

The next lemma provides a convenient criterion for comparing weighted averages associated with weights arising from Hardy fields. 
Indeed, eventual monotonicity follows from standard properties of Hardy field functions, while L'Hôpital's rule shows that condition \eqref{eq:equiv-weight} implies the corresponding hypothesis of Lemma~\ref{lem:KN:comparison}.
\begin{lem}
\label{lem:comp:weights}
Let $\mathbf{H}$ be a Hardy field.
For $i=1,2$, let $W_i(x)$ belong to $\mathbf{H}$ and assume that $w_i(x):= W_i'(x)$ is positive and non-increasing and that $\lim\limits_{x \rightarrow \infty} W_i(x) = \infty$.
Suppose that there exists $C > 0$ such that 
\begin{equation}
\label{eq:equiv-weight}
\lim_{x \rightarrow \infty} \frac{\log W_1(x)}{\log W_2(x)} \leq C.
\end{equation}
Let $(a_n)_{n \in \mathbb{N}}$ be a complex-valued sequence.
If 
\begin{equation*}
\lim_{N \rightarrow \infty} \frac{1}{W_2(N)} \sum_{n=1}^N w_2(n) a_n = a,
\end{equation*}
then 
\begin{equation*}
\lim_{N \rightarrow \infty} \frac{1}{W_1(N)} \sum_{n=1}^N w_1(n) a_n = a.
\end{equation*}
\end{lem}

We also recall the following form of Abel's summation formula. 
\begin{lem}[Abel's summation formula]\label{lem:byparts}
For any two sequences $(g_n)$ and $(h_n)$
\begin{equation*}
\sum_{k=m}^n g_k (h_{k+1}-h_k)= (g_{n+1} h_{n+1}-g_m h_m)-\sum_{k=m}^n h_{k+1}(g_{k+1}-g_k).
\end{equation*}
\end{lem}

We are now ready to prove the following auxiliary result, which is the main ingredient in the proof of Theorem~\ref{thm:sufficient}.
\begin{thm}
\label{thm:slow-weight}
Let $\mathbf{H}$ be a Hardy field and let $a(x) \in \mathbf{H}$ satisfy
\begin{equation*} 
\label{eq:cond:a}
\lim_{x \rightarrow \infty} a(x) = \infty \quad  \text{and} \quad \lim_{x \rightarrow \infty} \frac{a(x)}{x} = 0.
\end{equation*}
Suppose that $W(x) \in \mathbf{H}$ satisfies $w(x):= W'(x)$ is positive and non-increasing, $\lim\limits_{x \rightarrow \infty} W(x) = \infty$, 
and  there exists $c > 0$ such that for all sufficiently large $x$, 
\begin{equation*} W(x) \leq a(x)^c. \end{equation*}
If $(b_n)_{n \in \mathbb{N}}$ is a non-negative real sequence such that 
\begin{equation}
\label{eq:cond:b:lem}
\lim_{N \rightarrow \infty} \frac{1}{N} \sum_{n=1}^N b_n = b,
\end{equation}
then 
\begin{equation} 
\label{eq:conv:new}
\lim_{N \rightarrow \infty} \frac{1}{W(N)} \sum_{n=1}^N  w(n) b_{\lfloor a(n) \rfloor} = b 
\end{equation}
and 
\begin{equation} 
\label{eq:conv:new2}
\lim_{N \rightarrow \infty} \frac{1}{W(N)} \sum_{n=1}^N  w(n) b_{\lceil a(n) \rceil} = b. 
\end{equation}
\end{thm}

\begin{proof}
Let us first prove \eqref{eq:conv:new}.
Since changing finitely many values of $\lfloor a(n) \rfloor$ does not affect the limit in \eqref{eq:conv:new}, without loss of generality, we assume that $a(x) \geq 1$ for $x \geq 1$ and $a'(x)$ is positive and non-increasing for $x \geq 0$. 
By Lemma \ref{lem:comp:weights}, it is sufficient to prove \eqref{eq:conv:new} when $W(x) = a(x)$. 

Define two auxiliary functions $\gamma(x)$ and $\beta(x)$ such that 
\begin{itemize}
\item $\gamma(x)$ is the inverse function of $a(x)$, that is, $a\big(\gamma (x)\big)=x$ for all $x\in \setR_{\geq 0}$, 
\item  $\beta(x)=\int_{\gamma(x)}^{\gamma(x+1)} w(t) dt$.
\end{itemize}
Observe that $\beta(x) = W(\gamma (x+1) ) - W(\gamma(x)) = 1$, so it is non-increasing. 
Also, one can check that for each $k \in \mathbb{N}$
\begin{equation*}
\label{eq:estimate:beta}
 \beta(k) - w(a^{-1} (k)) \leq \sum_{\lfloor a(n) \rfloor =k} w(n) \leq \beta(k) + w(a^{-1} (k))
\end{equation*}
For each $N$, choose $K$ such that $K \leq a(N) < K+1$, so 
\begin{equation*}\label{eq:est1:seq}
 \sum_{k =1}^{ K} (\beta(k) - w(a^{-1} (k))) b_k 
\leq  \sum_{n=1}^{N} w(n) b_{\lfloor a(n) \rfloor}
\leq  \sum_{k=1}^{K+1} (\beta(k) + w(a^{-1} (k))) b_k,
\end{equation*}
and
\begin{equation}\label{eq:est2:seq}
 K = \sum_{k =1}^{ K} \beta(k) 
\leq  \sum_{n=1}^{N} w(n) 
\leq  \sum_{k=1}^{K+1} \beta(k) =K+1. 
\end{equation}
To finish proving \eqref{eq:conv:new} when $W(x) = a(x)$, it remains to show that 
\begin{equation}
\label{eq:last-lem}
 \lim_{K \rightarrow \infty} \frac{\sum_{k=1}^K w(a^{-1} (k)) b_k}{K} = 0.
\end{equation}
By \eqref{eq:cond:b:lem}, there exists $M$ such that $\sum\limits_{m=1}^n b_m \leq Mn$ for every $n \in \mathbb{N}$. 
Using summation by parts,  
\begin{align*}
\sum_{k=1}^K w(a^{-1} (k)) b_k 
&= \sum_{k=1}^{K-1} (w (a^{-1}(k)) - w (a^{-1}(k+1))) \cdot \sum_{m=1}^k b_m + w(a^{-1}(K))\sum_{m=1}^K b_m \\
&\leq \sum_{k=1}^{K-1} (w (a^{-1}(k)) - w (a^{-1}(k+1))) \cdot Mk + w(a^{-1}(K)) \cdot MK \\
&= M \sum_{k=1}^{K} w(a^{-1} (k)).
\end{align*}
Since $w(x)$ is decreasing and $a^{-1}(x) \geq x$, we have $w (a^{-1} (k)) \leq w (k)$. Therefore, using $K \leq a(N) < K+1$,
\[ \sum_{k=1}^K w (a^{-1} (k)) \leq \sum_{k=1}^K w (k) = \sum_{k=1}^{\lfloor a(N)\rfloor} w(k).\]
On the other hand \eqref{eq:est2:seq} gives $\sum_{k=1}^N w(n) \leq K+1$.
Finally,
\[  \frac{\sum_{k=1}^K w(a^{-1} (k))}{K} \leq \frac{K+1}{K} \frac{\sum_{k=1}^{\lfloor a(N) \rfloor} w(k)}{\sum_{k=1}^N w(n)}. \]
Furthermore, 
\[ \lim_{N \rightarrow \infty} \frac{\sum_{k=1}^{\lfloor a(N) \rfloor} w(k)}{\sum_{k=1}^N w(n)} = \lim_{N \rightarrow \infty} \frac{W(a(N))}{W(N)} = \lim_{N \rightarrow \infty} \frac{W'(a(N)) a'(N)}{W'(N)} = \lim_{N \rightarrow \infty} a'(a(N))  = 0. \]
Thus \eqref{eq:last-lem} follows.

The proof of \eqref{eq:conv:new2} is identical. Indeed, we have 
\begin{equation*}
 \beta(k) - w(a^{-1} (k)) \leq \sum_{\lceil a(n) \rceil =k} w(n) \leq \beta(k) + w(a^{-1} (k))
\end{equation*}
and the remainder of the proof is unchanged.
\end{proof}

\begin{proof}[Proof of Theorem~\ref{thm:sufficient}]
Let $(X, \mathcal{B}, \mu, T)$ be a probability measure preserving system. 
It is enough to show that for any $f \in L^1$ with $f \geq 0$, $\setA_N^{\mathsf{a}, W} f(x)$
converges almost everywhere. Moreover, if $T$ is ergodic, then the limit equals $\int f \,  d \mu$. 
Note that if $W(x) \in \mathbf{E}$ and $a(x) \in \mathbf{U}$, there exists a maximal Hardy field $\mathbf{H}$ containing both $W(x)$ and $a(x)$. 
Therefore, if $a(x)$ is eventually positive, by the pointwise ergodic theorem and \eqref{eq:conv:new} of Lemma \ref{thm:slow-weight}, we obtain the desired result.
Now suppose that $a(x)$ is eventually negative. Then
\[
T^{\lfloor a(n)\rfloor}
=(T^{-1})^{\lceil -a(n)\rceil}.
\]
Therefore, the desired conclusion follows by applying \eqref{eq:conv:new2} of Lemma~\ref{thm:slow-weight} to \(T^{-1}\).
\end{proof}

Next, we prove Theorem \ref{thm:necessary:int}.
\begin{thm}[Theorem \ref{thm:necessary:int}]
\label{thm:necessary}
Let $a(x) \in \mathbf{U}$ such that $1 \prec a(x) \prec x$. 
Let $W(x) \in \mathbf{E}$ be such that $w(x):= W'(x)$ is positive and non-increasing and $\lim_{x \rightarrow \infty} W(x) = \infty$. 
Suppose that there exists $c>0$ such that 
\begin{equation*}
W(x) \geq \exp (|a(x)|^{c})
\end{equation*}
for all sufficiently large $x$.
Then $(\lfloor a(n) \rfloor)_{n \in \mathbb{N}}$ is pointwise $L^{\infty}$-bad for $W$-averages: 
for every aperiodic measure preserving system $(X,\mathcal{B},\mu,T)$, there exists a function $f\in L^\infty$ such that the averages $\setA_N^{\mathsf{a}, W} f(x)$ fail to converge for almost every $x$. 
\end{thm}

To prove Theorem~\ref{thm:necessary}, we first reduce to the case where $a(x)$ is eventually positive.
Assume that $a(x)$ is eventually negative.
Note that if $t\notin\mathbb Z$, then $\lfloor t\rfloor=-\lfloor -t\rfloor-1.$
Moreover, the assumption $1\prec a(x) \prec x$ implies that
\[
\#\{1\le n\le N:\, a(n)\in\mathbb Z\}\le |a(N)|
\]
for all sufficiently large $N$.
Therefore, for every $f\in L^\infty$,
\[
\left|
\frac{1}{W(N)}\sum_{n=1}^N
w(n) \, f(T^{\lfloor a(n)\rfloor}x)
-
\frac{1}{W(N)}\sum_{n=1}^N
w(n) \, f\circ T^{-1}\big((T^{-1})^{\lfloor -a(n)\rfloor}x\big)
\right|
\le
\frac{|a(N)|}{W(N)}\|f\|_{L^\infty}
\rightarrow 0.
\]
Hence, throughout the remainder of this section, we may assume that $a(x)$ is eventually positive.

We need the following proposition for proving \Cref{thm:necessary}. 
\begin{prop}\label{prop:2}
        Let \(( \beta(n) )_{n=1}^\infty\) be a sequence of positive real numbers such that \\
    $\lim_{N}\sum_{n=1}^N \beta(n)=\infty$. Suppose that for every $\epsilon\in (0,1)$, \( L \in \setN \), and \( C>0 \), there exist a positive integer \( r \) and integers \( L<N_1<N_2<\cdots<N_r \ \) such that the following conditions hold:
\begin{equation}\label{eq:prop2}
\begin{aligned}
         (i)& \    \dfrac{\sum_{n=N_i}^{N_{i}+1} \beta(n)}{\sum_{n=1}^{N_{i+1}} \beta(n)}\geq 1-\epsilon \text{ for every } i\in [r].\\
    (ii)& \dfrac{N_r}{M}>C ,\text{ where }M:=\max \{N_i-N_{i-1}: i\in [r]\}\text{ with }\ N_0:=0.\end{aligned}\end{equation}
Then the pair $(\beta(n), n)$ satisfies the {\em strong sweeping out} property. 
That is, in every aperiodic measure preserving system $(X,\mathcal{B}, \mu, T)$, for each $\epsilon>0$ there exists $E\in \mathcal{B}$ with $\mu(E)<\epsilon$ such that for almost every $x\in X$,
\begin{equation*}
\limsup_{N\to \infty} \dfrac{1}{\sum_{n=1}^N \beta(n)}\sum_{n=1}^N \beta(n) \setone_E(T^n x)=1 \text{ and } \liminf_{N\to \infty} \dfrac{1}{\sum_{n=1}^N \beta(n)}\sum_{n=1}^N \beta(n) \setone_E(T^n x)=0.
\end{equation*}
In particular, $\displaystyle \lim_{N\to \infty}\dfrac{1}{\sum_{n=1}^N \beta(n)}\sum_{n=1}^N \beta(n) \setone_E(T^n x)$ fails to exist almost everywhere.
\end{prop}

To prove Proposition~\ref{prop:2}, we first introduce the necessary notation.
Let us define the upper density of a Lebesgue measurable set $A\subset \setR$ as follows:
\begin{equation*}
\overline{d}(A)=\limsup_{L\to \infty} \dfrac{1}{2L} \lambda(A\cap [-L,L]).
\end{equation*}
Let $(\beta(n))_{n=1}^\infty$ be as in Proposition~\ref{prop:2}.
Define the operators $\mathbbm{B}^\beta_N$ on locally integrable functions on $\setR$ as follows:
\begin{equation*}
\mathbbm{B}_N^\beta f(x)=\frac{1}{\sum_{n=1}^N \beta(n)}\sum_{n=1}^N \beta(n) f(x+n)
\end{equation*}

The following lemma is the key ingredient in the proof of Proposition~\ref{prop:2}.
\begin{lem}{}\label{lem:Rohlin}
 Let \((\beta(n))_{n=1}^\infty\) be a sequence of positive real numbers such that \\
    $\displaystyle\lim_{N}\sum_{n=1}^N \beta(n)=\infty$. Suppose that given any $\epsilon\in (0,1)$, any number $C>0$ and any positive integer $L$, there exists an integer $R>L$ and a Lebesgue measurable set $A\subset \setR$ such that
\begin{equation}\label{Rohlin1}
\overline{d}\left(\left\{x\in \setR: \sup_{L\leq N \leq R} \mathbbm{B}_N^\beta \setone_{{A}} (x)>1-\epsilon\right\}\right)>C \cdot \overline{d}(A).
\end{equation}
Then the conclusion of Proposition~\ref{prop:2} holds.
\end{lem}
\begin{proof}
This lemma follows immediately from \cite[Theorem 2.3]{SSO} by taking $\mathcal{W}=\setN$ and $\nu_n (k)=\setone_{\{1,2,\cdots, n\}}(k) \cdot \frac{\beta(k)}{\sum_{k=1}^n \beta(k)}.$
\end{proof}

\begin{proof}[Proof of Proposition~\ref{prop:2}] 
Let $L>0$ and $\epsilon\in (0,1)$ and $C>0$ be arbitrary. 
Without loss of any generality, we may assume that $C>8$. 
By Lemma~\ref{lem:Rohlin}, it is sufficient to show that \eqref{Rohlin1} holds. By hypothesis, there are positive integers $L<N_1< N_2< \cdots< N_r$ satisfying \eqref{eq:prop2}. 
Consider the periodic set $A:=\bigcup_{k\in \setZ} [k\cdot N_r-2M,k\cdot N_r]\subset \setR$. 
Suppose that $x\in (k-1)\cdot N_r+ [0, N_r-4M]$ for some $k\in \setZ$. Let $i$ be the largest index such that $x+N_{i-1}<k\cdot N_r - 2M$. Then $i$ must satisfy $2\leq i\leq r-1$. Also, for each $n\in [N_{i},N_{i+1}]$, we have $x+ n\in [k\cdot N_r-2M,k\cdot N_r]\subset  A$ .
Note that \[\overline{d}(A)=\dfrac{2M}{N_r}\leq \frac{2}{C}.\]

By hypothesis (i), we have 
\begin{equation*}
\mathbbm{B}_{N_{i+1}}^\beta \setone_{A} (x) \geq  \displaystyle\dfrac{\sum_{n=N_i}^{N_{i+1}}\beta(n)}{\sum_{n=1}^{N_{i+1}}\beta(n)}\geq 1-\epsilon,
\end{equation*} 
and by hypothesis (ii), we obtain
\begin{align*}
\overline{d}\left(\left\{x\in \setR: \sup_{N_1\leq N \leq N_r} \mathbbm{B}_{N}^\beta \setone_{A} (x)\geq 1-\epsilon\right\}\right) 
&= \dfrac{N_r-4M}{N_r} = 1-4\frac{M}{N_r} \\ 
&\geq 1-\frac{4}{C} > \frac{1}{2} = \frac{C}{4}\cdot \frac{2}{C}\geq \frac{C}{4}\cdot \overline{d}(A).
\end{align*}
Since $C$ is arbitrary, \eqref{Rohlin1} follows.
\end{proof}

\begin{proof}[Proof of \Cref{thm:necessary}]
Suppose ${\log W(x)}\gg [a(x)]^\tau \text{ for some }\tau>0.$ 
By Lemma~\ref{lem:comp:weights}, it is sufficient to show that the conclusion holds for $W_1(x)= e^{a(x)^\tau}$ for some $\tau>0$. 
Applying the lemma once more, we may further assume that $\tau<\frac{1}{2}$. 
For convenience, we rename $W_1(x)$ as $W(x)$. 
Let $w (x)=W'(x)$, and let $\gamma(y)$ be the compositional inverse of $a(y)$.
Define 
$$\alpha(k)= \sum_{l \in \{ [\gamma (k),\gamma(k+1))\cap \setN \}} w (l)$$  and 
$$ \beta(y)=\int_{\gamma(y)}^{\gamma(y+1)} w(x) dx=e^{(y+1)^\tau}-e^{y^\tau}.$$
For $E\in\mathcal{B}$ and $x\in X$, define
\[
\mathbb{D}^{\alpha}_N \setone_E(x)
=
\dfrac{1}{\sum_{k=1}^{\lfloor a(N)\rfloor}\alpha(k)}
\sum_{k=1}^{\lfloor a(N)\rfloor}
\alpha(k)\setone_E(T^k x).
\]
By Remark~\ref{rem:range}, it is clear that $\lim_{N\to \infty}\mathbb{D}^{\alpha}_N  \setone_E(x)$ exists \emph{iff} $\lim_{N\to \infty}\mathbb{A}^{\mathsf{a},W}_N \setone_E(x)$ exists. So, it is sufficient to consider $(\mathbb{D}^\alpha_n)_{n \in \setN}$.

Furthermore, using \begin{equation*}\label{ques}
|\alpha(k)-\beta(k)| \leq w (a^{-1}(k)),
\end{equation*}
we obtain
\begin{align*}
&\dfrac{1}{\sum_{k=1}^{\lfloor a(N)\rfloor}\ \alpha(k)} \left|\sum_{k=1}^{\lfloor a(N)\rfloor } \alpha(k) \setone_E (T^{k}x)-\sum_{k=1}^{\lfloor a(N)\rfloor } \beta(k) \setone_E (T^{k}x)\right| \\
& \quad \quad \leq \dfrac{1}{\sum_{k=1}^{\lfloor a(N)\rfloor}\alpha(k)} \sum_{k=1}^{\lfloor a(N)\rfloor } |\alpha(k)- \beta(k)| 
\leq  \dfrac{\sum_{k=1}^{\lfloor a(N)\rfloor} w (a^{-1}(k))}{\sum_{k=1}^{\lfloor a(N)\rfloor}\alpha(k)} \\
& \quad \quad \leq \dfrac{\sum_{k=1}^{\lfloor a(N)\rfloor} w (k)}{\sum_{k=1}^{\lfloor a(N)\rfloor}\alpha(k)} \leq \frac{W (a(N) +1)}{W(N)} \rightarrow 0.
\end{align*}
Also, $$\lim_{N\to\infty}\dfrac{\sum_{k=1}^{\lfloor a(N)\rfloor}\beta(k)}{\sum_{k=1}^{\lfloor a(N)\rfloor}\alpha(k)}=1.$$
Thus, $\mathbb{D}^\alpha_N \setone_E(x)$ and $\dfrac{1}{\sum_{k=1}^N \beta(k)}\sum_{k=1}^N \beta(k) \setone_E (T^k x)$ have the same asymptotic behavior. Hence, it suffices to show that $(\beta(n),n)$ satisfies the strong sweeping out property. To establish this, we will apply Proposition~\ref{prop:2}. Let $L\in \setN$, $\epsilon\in (0,1)$ and $C>0$ be arbitrary. 
Let $r$ be a positive integer to be specified later.
For every $i\in [r]$, define $s=4/\tau \text{ and }T_i=(L+i)^{s}$.
Observe that 
\[\lim_{i\to \infty}\dfrac{\int_{T_i}^{T_{i+1}}\beta(x) dx}{\int_{1}^{T_{i+1}}\beta(x) dx}=\lim_{y\to \infty}\dfrac{\int_{y^s}^{(y+1)^s}\beta(x) dx}{\int_{1}^{(y+1)^s}\beta(x) dx}.\]

We apply L'H\^opital's rule to compute the limit of the expression on the right.
\begin{align*} \lim_{y\to \infty}\dfrac{\int_{y^s}^{(y+1)^s}\beta(x) dx}{\int_{1}^{(y+1)^s}\beta(x) dx}&= \lim_{y\to \infty} \dfrac{s\cdot \beta\big((y+1)^s\big) \cdot (y+1)^{s-1}-s\cdot \beta\big(y^s\big)\cdot  (y)^{s-1}}{s\cdot \beta\big((y+1)^s\big) \cdot (y+1)^{s-1}}\\
&\geq \lim_{y\to \infty} \left(\dfrac{y}{y+1}\right)^{s-1} \left(1-\dfrac{\beta\big(y^s\big)}{\beta\big((y+1)^s\big)}\right)\\
&= \lim_{y\to \infty} \left(\dfrac{y}{y+1}\right)^{s-1}\cdot \lim_{y} \left(1-\dfrac{\beta\big(y^s\big)}{\beta\big((y+1)^s\big)}\right)\\
&=\lim_{y\to \infty} \left(1-\dfrac{\beta\big(y^s\big)}{\beta\big((y+1)^s\big)}\right)= 1.
\end{align*}

Hence, there exists $I\in \setN$ such that
\begin{equation*}
\dfrac{\sum_{n=T_i}^{{T_{(i+1)}}}\beta(n)}{\sum_{n=1}^{{T_{(i+1)}}}\beta(n)} \geq 1-\epsilon \text{ for all } i\geq I.
\end{equation*}
Letting $N_k=T_{I+k}$ for $k=1,2,\cdots, r$, we see that condition $(i)$ is satisfied.

Let $L'=L+I$. We see that 
\begin{align*}
\dfrac{N_r}{\max_{i\in [r-1]}(N_{i+1}-N_i)}&=  \dfrac{(L'+r)^s}{(L'+r)^s-(L'+r-1)^s} \to \infty \text{ as } r\to \infty,\\
\text{ and }\dfrac{N_r}{N_1}&= \dfrac{(L'+r)^s}{(L'+1)^s} \to \infty \text{ as } r\to \infty.
\end{align*}
Thus, choosing $r$ large enough, we can ensure that $\dfrac{N_r}{M}>C$, and the condition $(ii)$ is also satisfied. 
This completes the proof of Theorem~\ref{thm:necessary}.
\end{proof}

\section{Sparse sequences}
\label{sec:sparse1}
In this section, we prove the following result.
\begin{thm}\label{thm:sparse2}
Let $W(x) \in \mathbf{E}$ be such that $w(x) = W'(x)$ is positive and non-increasing and $\lim_{x \rightarrow \infty} W(x) = \infty$. 
Let $a(x)\in \mathbf{U}$ be a subpolynomial function with the canonical decomposition $a(x)=p(x)+r(x)$, where  $d:=\deg (p) \geq 2$. 
Suppose $|r(x)|\gg \log^{\eta} x$, for some $\eta>1$. Set $s:=2^{d+1}$. 
Assume that
\begin{align}
\label{eq3.1:condition:eta-W}
(\log t)^{\eta-1}\succ  \left(\log W(t)\right)^{s-2+\epsilon} \text{ for some }\epsilon>0.
\end{align}

Then, $(\lfloor a(n) \rfloor)_{n \in \mathbb{N}}$ is pointwise $L^2$-good and ergodic for $W$-averages:
For any system $(X,\mathcal{B},\mu,T)$ and any $f\in L^2(\mu)$, 
\begin{equation}
\label{eq3.3}
\lim_{N\to \infty}\setA_N^{\mathsf{a}, W}f(x) \text{ exists for almost every} \ x\in X.
\end{equation}
Moreover, if the system is ergodic, then the limit is $\int_X f \, d \mu.$
\end{thm}

Note that if we take $w(n)=1$ for all $n$, then the expression in \eqref{eq3.3} reduces to the classical Ces\`aro averages and in this case we recover Theorem 3.5 of \cite{BKQW} which states the following:
\begin{cor}\label{Thm9.2:BKQW}
Let $a(x)\in \mathbf{U}$ be a subpolynomial function with the canonical decomposition $a(x)=p(x)+r(x)$, where  $d := \deg (p) \geq 2$. 
If for some $\epsilon>0$, $r(x)\succ (\log x)^{2^{d+1}-1+\epsilon}$, then $(\lfloor a(n) \rfloor)$ is pointwise $L^2$-good and ergodic.
\end{cor}

\begin{cor}\label{cor:logcase}
Let $a(x)\in \mathbf{U}$ be a subpolynomial function with the canonical decomposition $a(x)=p(x)+r(x)$ and $\deg (p) \geq 2$. 
If $|r(x)|\gg \log^{\eta} x$, for some $\eta>1$, then $(\lfloor a(n)\rfloor)$ is pointwise $L^2$-good for $\log$-averages.
\end{cor}
\begin{rem}\label{rem:log}
The above corollary shows that any sequence $(\lfloor a(n) \rfloor )_{n \in \mathbb{N}}$ that satisfies the hypothesis of Theorem~\ref{thm:sparse2} is pointwise $L^2$-good for logarithmic averages. Let $1\leq p\leq \infty$. From Lemma~\ref{lem:comp:weights}, we know that if a sequence $(\lfloor a(n)\rfloor)$ is pointwise $L^p$-good for $W_1$-average and $W_2$ is a weaker averaging method than $W_1$, that is, $W_2\ll W_1$, then $(\lfloor a(n)\rfloor)$ is also pointwise $L^p$-good for $W_2$-average. Hence, Theorem~\ref{thm:sparse2} does not provide any additional information when one considers a weaker averaging method than logarithmic averages. 
\end{rem}

Theorem~\ref{thm:sparse2} is proved using Lemma~\ref{lem:reduction1}, Lemma~\ref{lem:reduction2} and Proposition~\ref{prop:exponential}. Proposition~\ref{prop:exponential} is proved in Subsection~\ref{sub:proposition}.  \emph{For the rest of this section, let us fix a measure preserving system $(X,\mathcal{B},\mu,T)$, and assume that $\log t \ll  W(t)\ll t$, which is enough in view of Remark~\ref{rem:log}.}

\subsection{Proof of Theorem~\ref{thm:sparse2}}\label{sub:mainargument} 
For convenience, we modify our notation and introduce some additional notation.

For a real-valued function $a(x)$, a function $f\in L^2(\mu)$ and positive real number $t$ 
\begin{equation*}\label{eq:notaion}
\begin{aligned}
&\setA^{\mathsf{a}}_tf(x):= \frac{1}{t} \sum_{n\leq t} f(T^{\lfloor a(n) \rfloor}x),\ \  \setA^{\mathsf{a},W}_tf(x):= \frac{1}{W (t)} \sum_{n\leq t} w(n) f(T^{\lfloor a(n) \rfloor}x),\\
&\setB^{\mathsf{a}}_tf(x):= \frac{1}{t} \sum_{n\leq t} f(T^{\lfloor p(n)+\log^s n\rfloor}x),\ \  \setB^{\mathsf{a},W}_tf(x):= \frac{1}{W(t)} \sum_{n\leq t} w(n) f(T^{\lfloor p(n)+\log^s n\rfloor}x).
\end{aligned}
\end{equation*}
The Fourier transform of the operator $\setA^\mathsf{a}_t$ is given by  
$$\displaystyle\hat{\setA}^\mathsf{a}_t (\beta)=\frac{1}{t}\sum_{n\leq t} e\left( \lfloor a(n) \rfloor\beta\right).$$ 
The Fourier transform for each of the operators $\setA^{\mathsf{a},W}_t, \setB^{\mathsf{a}}_t$ and $\setB^{\mathsf{a},W}_t$ also has an analogous expression.

\begin{rem} \begin{enumerate}
    \item Note that the operators $\setB^{\mathsf{a}}_t$ and $\setB_t^{\mathsf{a},W}$ depend only on the polynomial part $p(x)$ of $a(x)$. Lemma~\ref{lem:reduction1} explains why we deal with the operator $\setB_t^{\mathsf{a},W}$.
    \item Based on the method of \cite{BKQW}, it is natural to compare $\setA^{\mathsf{a},W}_t f(x)$ with \linebreak  $\frac{1}{W(t)} \sum_{n\leq a(t)} w(n)\cdot {F'}(n)\cdot f(T^{n}x)$ instead of $\setB^{\mathsf{a},W}_t f(x)$, where $F(x)$ is the compositional inverse of $a(x)$.  However, this choice leads to some difficulties which we are able to resolve by comparing $\setA^{\mathsf{a},W}_t f(x)$ with $\setB^{\mathsf{a},W}_t f(x)$.
\end{enumerate} 
\end{rem}

\begin{lem}\label{lem:reduction1} Let $a(x)$ and $r(x)$ be as in Theorem~\ref{thm:sparse2}, and $|r(x)| \gg  \log^ s x$. Then
for any function $f\in L^2(\mu), \displaystyle\lim_{t\to \infty}\setA^{\mathsf{a},W}_tf(x) $
 exists for almost every $x\in X$. Moreover, if the system is ergodic then the limit is $\int_X f d\mu.$
 In particular, the same conclusions hold for the operator $\setB_t^{\mathsf{a},W}$.
\end{lem}

\begin{proof}
Let $f\in L^2(\mu)$. By \cite[Theorem 3.5]{BKQW}, we know that $\displaystyle\lim_{t\to \infty}\setA^{\mathsf{a}}_tf(x)$ exists for almost every $x\in X$, and if the system is ergodic then the limit is $\int_X f \, d\mu$. The conclusion then follows from Lemma~\ref{lem:comp:weights}.
\end{proof}

We will need the following modification of the \emph{lacunary trick}.
 \begin{lem}\label{lem:reduction2}
 Let $\tilde{W}(x)$ be the compositional inverse of $W(x)$. Let $g,h \in L^1(\mu)$, $g\geq 0$.   For \( \rho \in \setR \) and \( i \in \setN \), set \( N_i
        = N_i(\rho):=
        \tilde{W}{(\rho^i)}.
        \) 
    Suppose that for every $\rho>1$, there exists a set $X(\rho)$ of full measure such that
        \[ 
            \lim_{i \to \infty }\setA^{\mathsf{a},W}_{N_i}g(x)= h(x) \text{ for every \( x\in X(\rho) \).}
        \]
         Then 
        \[ \lim_{t \to \infty }\setA^{\mathsf{a},W}_tg(x)= h(x) \text{ for almost every }x\in X.
        \] 
    \end{lem}

\begin{proof}
 Let $(\rho_k)_{k \in \mathbb{N}}$ be a sequence which decreases to $1$.  Consider the set $X'=\bigcap_{k=1}^{\infty} X(\rho_k)$ of full measure. For a fixed $\rho_k$, consider $N_i=N_i(\rho_k)$. Let \( N \) be such that \( N_i \leq N < N_{i+1} \).  Observe that for $x\in X'$,
        \begin{align*}
            \frac{1}{\rho_k} \cdot {\setA^{\mathsf{a},W}_{ N_i}} g(x) 
            \leq 
            {\setA^{\mathsf{a},W}_{ N}} g(x)
            \leq 
            \rho_k \cdot {\setA^{\mathsf{a},W}_{ N_{i+1}}} g(x). 
        \end{align*}
        This implies
        \begin{align*}
            \frac{1}{\rho_k} \cdot \lim_{i \to \infty} {\setA^{\mathsf{a},W}_{ N_i}} g(x)  
            \leq 
            \liminf_{N \to \infty}  {\setA^{\mathsf{a},W}_{ N}} g(x) 
            \leq 
            \limsup_{N \to \infty} {\setA^{\mathsf{a},W}_{ N}} g(x) 
            \leq 
            \rho_k \cdot \lim_{i \to \infty} {\setA^{\mathsf{a},W}_{ N_i}} g(x) . 
        \end{align*}
        By assumption, \( \lim_{i \to \infty} {\setA^{\mathsf{a},W}_{ N_i}} g( x) = h(x) \). Hence, we have
        \begin{equation*}
            \frac{h(x)}{\rho_k} \leq \liminf_{N}{\setA^{\mathsf{a},W}_{ N}} g(x) 
            \leq 
            \limsup_{N}{\setA^{\mathsf{a},W}_{ N}} g( x)
            \leq 
            \rho_k h(x).
        \end{equation*}
        Letting $k \to \infty$, the result follows.
\end{proof}

\begin{prop}{}\label{prop:exponential}
Let $W(x)$ and $a(x)$ be as in Theorem~\ref{thm:sparse2}. Further, assume that $a(x)$ has the canonical decomposition $a(x) = p(x) + r(x)$ with $d\geq 2$ and $\log^s x\gg|r(x)|\gg (\log x)^\eta$ for some $\eta>1$. Then there exist positive constants $\epsilon, C, t_0$ such that for all  $\beta$ with $0\leq  |\beta|\leq \frac{1}{2}$ and $t\geq t_0$ we have
\begin{equation}\label{eq:toshow22}
\left| \hat{\setA}^{\mathsf{a},W}_{t}(\beta)-\hat{\setB}^{\mathsf{a},W}_{t}(\beta)\right|\le \frac{C}{\left(\log W(t)\right)^{\frac{1}{2}+\epsilon}}.
\end{equation}
\end{prop}

The proof of Proposition~\ref{prop:exponential} is postponed until the end of this section.

\begin{proof}[Proof of Theorem~\ref{thm:sparse2}]
Let $f\in L^2(\mu)$. Without loss of any generality, we can assume that $f\geq 0$. If $|r(x)|\gg \log^s x$, then by Lemma~\ref{lem:reduction1}, we have the desired conclusion. So, let us assume that $\log^s x\gg|r(x)|\gg \log^{\eta} x$ for some $\eta>1$.
By Lemma~\ref{lem:reduction1}, it is enough to show that
\begin{equation*}
\lim_{t\to \infty}\left|\setA^{\mathsf{a},W}_tf(x)-\setB^{\mathsf{a},W}_tf(x)\right|= 0 \text{ for almost every } x\in X.
\end{equation*}

Let $\rho>1$ be arbitrary and $N_i=\tilde{W}{(\rho^i)}.$ By Lemma~\ref{lem:reduction2}, it will be sufficient to show that 
\begin{equation*}
\lim_{i\to \infty}\left|\setA^{\mathsf{a},W}_{N_i}f(x)-\setB^{\mathsf{a},W}_{N_i}f(x)\right|= 0 \text{ for almost every } x\in X.
\end{equation*}
This will follow if we can show that
\begin{equation}\label{eq:toshow0}
\int_{X}\sum_{i\in \setN} \left|\setA^{\mathsf{a},W}_{N_i}f(x)-\setB^{\mathsf{a},W}_{N_i}f(x)\right|^2 d\mu= \sum_{i\in \setN}\int_{X} \left|\setA^{\mathsf{a},W}_{N_i}f(x)-\setB^{\mathsf{a},W}_{N_i}f(x)\right|^2 d\mu <\infty.
\end{equation}
By the spectral theorem, there exists a spectral measure $\mu_f$ such that
\begin{equation}\label{eq:spectral1}
\sum_{i\in \setN}\int_{X} \left|\setA^{\mathsf{a},W}_{N_i}f(x)-\setB^{\mathsf{a},W}_{N_i}f(x)\right|^2 d\mu=\sum_{i\in \setN}\int_{\setT} \left| \hat{\setA}^{\mathsf{a},W}_{N_i}(\beta)-\hat{\setB}^{\mathsf{a},W}_{N_i}(\beta)\right|^2 d\mu_f.
\end{equation}
By Proposition~\ref{prop:exponential}, we have
\begin{equation*}\label{eq:toshow}
\left| \hat{\setA}^{\mathsf{a},W}_{t}(\beta)-\hat{\setB}^{\mathsf{a},W}_{t}(\beta)\right|\ll \frac{1}{\left(\log W(t)\right)^{\frac{1}{2}+\epsilon}} \text{ for all } \beta \text{ with } 0\leq |\beta|\leq \frac{1}{2}.
\end{equation*}
 Hence, there exists $C=C(\rho)>0$ such that
\begin{equation*}\label{eq:spectral}
\left| \hat{\setA}^{\mathsf{a},W}_{N_i}(\beta)-\hat{\setB}^{\mathsf{a},W}_{N_i}(\beta)\right|\ll \frac{C}{i^{\frac{1}{2}+\epsilon}} \text{ for all } \beta \text{ with } 0\leq |\beta|\leq \frac{1}{2}.
\end{equation*}

This implies that the right-hand side of \eqref{eq:spectral1} is finite. This establishes \eqref{eq:toshow0}, thereby completing the proof of the theorem.
\end{proof}

\subsection{Proof of Proposition~\ref{prop:exponential}}\label{sub:proposition} 
The proof of Proposition~\ref{prop:exponential} is divided into two parts:
for small values of $\beta$, we prove \eqref{eq:toshow22} in Lemma~\ref{lem:smallregime}, whereas for large values of $\beta$, we prove it in Lemma~\ref{lem:largeregime}.

\begin{lem}\label{lem:smallregime}
Let $a(x)$ and $W(x)$ be as in Proposition~\ref{prop:exponential}. 
Then there exists $C>0$ such that for $0\le |\beta|\le \frac{1}{\log^{s+1} t}$,
\begin{equation*}
\left| \hat{\setA}^{\mathsf{a},W}_{t}(\beta)-\hat{\setB}^{\mathsf{a},W}_{t}(\beta)\right|\leq \frac{C}{\left(\log W(t)\right)^{\frac{3}{4}}}.
\end{equation*} 
\end{lem}

\begin{proof}
Suppose $\beta$ satisfies the hypothesis of the lemma. Then
\begin{align}
\left| \hat{\setA}^{\mathsf{a},W}_{t}(\beta)-\hat{\setB}^{\mathsf{a},W}_{t}(\beta)\right|&= \left|\frac{1}{W(t)}\sum_{n\leq t} w(n) \left( e\left( \lfloor a(n) \rfloor\beta\right)-e\left( \lfloor p(n)+\log^s n\rfloor\beta\right)\right)\right| \notag\\
&\leq \frac{1}{W(t)}\sum_{n\leq t} w(n) \left|e\left( \lfloor a(n) \rfloor\beta\right)-e\left( \lfloor p(n)+\log^s n\rfloor\beta\right)\right|  \notag\\
&\leq \frac{1}{W(t)}\sum_{n\leq t} w(n) 2 \pi |\beta| (\log^s n + |r(n)| + 1)  \notag\\
&\leq \frac{5\pi |\beta| \log^{s}t}{W (t)}\sum_{n\leq t} w(n) \label{eq:31}\\
&\leq {5\pi |\beta|}{\log^s t}  \notag\\
&\ll \frac{1}{\left(\log W(t)\right)^{3/4}}\label{eq:32}.
\end{align}
We used the inequality $|e(x)-e(y)|\leq 2\pi |x-y|$ and the fact that $\eta\leq s$ to obtain \eqref{eq:31}, and $W(t)\leq t$ for large $t$ to obtain \eqref{eq:32}.
\end{proof}

To establish the estimate for large values of $\beta$, we will use the following lemmas.
\begin{lem} \label{lem:Lopital}
Let $r(x)\in \mathbf{U}$ with $1 \prec r(x) \prec x$. Let $a>1$ be a real number.
\begin{enumerate}
\item If $(\log x)^a \prec r(x)$, then for any positive integer $k$, 
\begin{equation*}\label{eq:germ}
\dfrac{(\log x)^{a-1}}{x^k} \prec r^{(k)}(x).
\end{equation*}
 \item If  $r(x)\prec (\log x)^a$,  then for any positive integer $k$
\begin{equation*}\label{eq:germ2}
r^{(k)}(x)\prec \dfrac{(\log x)^{a-1}}{x^k}.
\end{equation*}
\end{enumerate}
\end{lem}
\begin{proof} We will prove (1). The proof of (2) is analogous and omitted. 
The proof is by induction. Let $\displaystyle\lim_{x\to \infty}\dfrac{r(x)}{(\log x)^a}=\pm\infty$. Then by L'Hôpital's rule, we have
\begin{align*}
\lim_{x\to \infty}\dfrac{r'(x)}{\frac{(\log x)^{a-1}}{x}}=\pm\infty.
\end{align*}
Suppose the conclusion is true for $k-1$ for some $k\geq 2$. So, 
\begin{equation*}
\lim_{x\to \infty}\dfrac{r^{(k-1)}(x)}{\frac{(\log x)^{a-1}}{x^{k-1}}}=\pm\infty.
\end{equation*}
Observe that since $r(x)\prec x$, by L'Hôpital's Rule, $\lim_{x\to \infty} r^{(k-1)}(x)=0$ for all $k\geq 2$. Hence, applying L'Hôpital's Rule again, we obtain
\begin{equation*}
\lim_{x\to \infty}\dfrac{r^{(k)}(x)}{\frac{(\log x)^{a-1}}{x^{k}}}=\pm\infty.
\end{equation*} 
This proves the desired conclusion.
\end{proof}

\begin{lem}\label{lem:van der Corput}[Van der Corput \cite{vdc}] 
Let $Y,X, l$ be integers. Suppose that $Y<X, l\geq 2$ and set $u=2^l$. Suppose that the real function $\phi(x)$ is $l$-times differentiable in the interval $[Y,X]$ and $|\phi^{(l)}(x)|\geq \rho$ in $[Y,X]$, where $\rho$ is a positive number. Also, let 
\[R=\frac{1}{X-Y}\left|\phi^{(l-1)}(Y)-\phi^{(l-1)}(X)\right|.\]
Then we have
\begin{equation*}
\left|\sum_{Y\leq n\leq X}e\left(\phi(n)\right)\right|\leq 21 (X-Y) \left( \left(\dfrac{R^2}{\rho}\right)^{1/(u-2)}+\left(\dfrac{1}{\rho(X-Y)^l}\right)^{2/u}+\left(\dfrac{R}{\rho(X-Y)}\right)^{2/u}  \right).
\end{equation*}
\end{lem}

\begin{lem}\label{lem:largeregimereal} Let $a(x)$ and $W(x)$ be as in Proposition~\ref{prop:exponential}. There exist a sufficiently small $\epsilon > 0$ and a constant $C >0$ such that 
for $\frac{1}{t^{d-\frac{1}{2}}}\leq |\beta| \leq t^\epsilon$, we have
\begin{equation}\label{eq:toshow1}
\frac{1}{t}\left|\sum_{n\leq t} e\left( a(n)\beta\right)\right|\leq  \frac{C}{\left(\log W(t)\right)^{1+3\epsilon}}.
\end{equation}
\end{lem}
 In the proof of this lemma, the implicit constant depends only on $a(x), W(x)$ and $\epsilon$. 
\begin{proof}
Recall that $d=\text{deg}(p)\geq 2$ and $s=2^{d+1}$. Let $\epsilon \in (0,1)$ be a small number to be determined later.

\noindent\textbf{Case (I).} Let $\frac{1}{t^{d-\frac{1}{2}}}\leq |\beta|\leq \left(\frac{1}{\left(\log W(t)\right)^{1+3\epsilon}}\right)^{\frac{s}{2}-2}$.\\
In this case, we apply Lemma~\ref{lem:van der Corput} to the function $\phi(x)= a(x)\beta$ with $l=d, X=t$, and $Y=t^{3/4}$.
We have \(\phi^{(l)}(x)= \left(A + r^{(l)}(x)\right)\beta\), where $A$ is a nonzero constant depending only  on  $a(x)$. 
Since $r^{(l)} (x)\to 0$ as $x\to \infty$, for sufficiently large $t$, $|\phi^{(l)}(x)|\geq \frac{1}{2}|A||\beta|$ on $[t^{3/4},t].$ So, $\rho$ can be taken to be $\dfrac{1}{2}|A||\beta|$. 
By the mean value theorem, $R\leq 3 |A| |\beta|$. Hence $\frac{R^2}{\rho}\leq 36|A||\beta|$, $\frac{1}{\rho(X-Y)^l}\leq \frac{2^{d}}{|A||\beta| t^{d}}$ and $\frac{R}{\rho(X-Y)}\leq \frac{2}{t}$. 
Hence, for sufficiently large $t$, by Lemma~\ref{lem:van der Corput}, in the given range of $\beta$ we have 
\begin{equation*}
\left|\sum_{t^{3/4}\leq n\leq t} e\left( a(n)\beta\right)\right| \ll  (t-t^{3/4}) \frac{1}{\left(\log W(t)\right)^{1+3\epsilon}}.
\end{equation*}
This implies that for sufficiently large $t$, we have
\begin{equation*}\label{eq:case1}
\frac{1}{t}\left|\sum_{n\leq t} e\left( a(n)\beta\right)\right| \ll \frac{t^{3/4}}{t}+ \dfrac{(t-t^{3/4})}{t}\frac{1}{\left(\log W(t)\right)^{1+3\epsilon}} \ll \frac{1}{\left(\log W(t)\right)^{1+3\epsilon}}.
\end{equation*} 

\noindent\textbf{Case (II).} Let $\frac{\left(\log W(t)\right)^{\frac{s}{2}(1+3\epsilon)}}{(\log t)^{\eta-1}} \leq |\beta|\leq t^\epsilon$.\\
In this case, we apply Lemma~\ref{lem:van der Corput} to the function $\phi(x)= a(x)\beta$ with $l=d+1,  X=t ,Y=t^{1-\tau}$, where $\tau=\frac{1}{4(d+1)}$. Since $|r(x)|\gg \log^\eta x$, $r(x)\succ \log^{\eta'}x$ for some $\eta'>1$ (for simplicity we write $\eta'=\eta$). Hence, by Lemma~\ref{lem:Lopital}, we have
\begin{equation*}
|\phi^{(d+1)}(x)|= |r^{(d+1)}(x) \beta| \succ \left(\dfrac{(\log x)^{\eta-1}}{x^{d+1}}\right)|\beta|.
\end{equation*}
So, we can take $\rho= \dfrac{(\log t)^{\eta-1}}{t^{d+1}}|\beta|$. Using again the mean value theorem, Lemma~\ref{lem:Lopital} and the fact that $|r(x)|\ll \log^ s x$, for sufficiently large $t$, we have $R\leq \dfrac{C (\log t)^{s-1}}{t^{(d+1)(1-\tau)}}|\beta|$, for some $C>0$. 
\begin{equation*}
\frac{R^2}{\rho}\leq \frac{C^2 (\log t)^{2s-\eta-1}}{t^{(d+1)(1-2\tau)}}|\beta|, \frac{1}{\rho(X-Y)^l}\leq \frac{C}{|\beta| (\log t)^{\eta-1}}\text{ and }\frac{R}{\rho(X-Y)}\leq \frac{C (\log t)^{s-\eta}}{t^{1-(d+1)\tau}}.
\end{equation*}
In the given range, we have $\frac{1}{|\beta| (\log t)^{\eta-1}}\leq \frac{1}{\left(\log W( t)\right)^{\frac{s}{2}(1+3\epsilon)}}$.

Since $(d+1)(1-2\tau)={d+1}-\frac{1}{2}>1>\epsilon$, by Lemma~\ref{lem:van der Corput}, we have
\begin{equation*}
\left|\sum_{t^{1-\tau}\leq n\leq t} e\left( a(n)\beta\right)\right| \ll  (t-t^{(1-\tau)}) \frac{1}{\left(\log W( t)\right)^{1+3\epsilon}}.
\end{equation*}
As in Case (I), for sufficiently large $t$,
\begin{equation*}\label{eq:case2}
\frac{1}{t}\left|\sum_{n\leq t} e\left( a(n)\beta\right)\right| \ll \frac{1}{\left(\log W( t)\right)^{1+3\epsilon}}.
\end{equation*}
 Finally, by condition \eqref{eq3.1:condition:eta-W}, we may choose $\epsilon\in (0,1)$ sufficiently small so that $\frac{\left(\log W(t)\right)^{\frac{s}{2}(1+3\epsilon)}}{(\log t)^{\eta-1}}\leq \left(\frac{1}{\left(\log W(t)\right)^{1+3\epsilon}}\right)^{\frac{s}{2}-2}$ for all sufficiently large $t$. 
Hence, combining Case (I) and Case (II), the conclusion of the lemma follows.
\end{proof}

\begin{prop}\label{prop:integerpart} Let $a(x)$ and $W(x)$ be as in Proposition~\ref{prop:exponential}.
Then there exist positive constants $t_0, C$ and $\epsilon$ such that for all $t\geq t_0$ and $\beta$ with $\frac{1}{t^{d-\frac{1}{2}}}\leq |\beta| \leq \frac{1}{2}$, we have
\begin{equation}\label{eq:toshow2}
\frac{1}{t}\left|\sum_{n\leq t} e\left( \lfloor a(n)\rfloor \beta\right) \right|\leq \frac{C}{\left(\log W(t)\right)^{\frac{1}{2}+\epsilon}}.
\end{equation}
The constants $C$ and $ t_0$ depend only on $a(x), W(x)$ and $\epsilon$.
\end{prop}
\begin{rem}
Note that the exponential sums that appear in \eqref{eq:toshow1} and \eqref{eq:toshow2} are different, the latter involves with the integer part of $a(n)$. The following lemma enables us to obtain an estimate of the type \eqref{eq:toshow2} from an estimate of the type \eqref{eq:toshow1}.
\end{rem}

\begin{lem}\label{lem:integer part}{\cite[Theorem~5.2]{BKQW}}
Suppose that $(\phi(n))$ is a sequence of real numbers, the numbers $t, Q, r, S$ are all greater than $10$ and $r^2\leq S-1$. If 
\begin{equation*}
\frac{1}{ t}\left|\sum_{n\leq t} e(\phi(n)\alpha)\right| \leq  \frac{1}{r} \text{ for each }\frac{1}{Q}\leq |\alpha| \leq S, 
\end{equation*}
then for every $\alpha$ with $\frac{1}{Q}\leq |\alpha|\leq \frac{1}{2}$, we have
\begin{equation*}
\frac{1}{t}\left|\sum_{n\leq t} e(\lfloor \phi(n)\rfloor\alpha)\right|\leq 100\sqrt{\dfrac{\log r}{r}}.
\end{equation*}
\end{lem}

\begin{proof}[Proof of Proposition~\ref{prop:integerpart}] We apply Lemma~\ref{lem:integer part} with $\phi(n)= a(n)$, $Q= t^{d-\frac{1}{2}}$, $r= \left(\log W(t)\right)^{1+3\epsilon}$ and $S= t^{\epsilon}$. Clearly, for sufficiently large $t$ (depending only on $a(x), W(x)$ and $\epsilon$) we have $r^2\leq S-1$. By Lemma~\ref{lem:largeregimereal}, our choice of parameters satisfies the hypothesis of Lemma~\ref{lem:integer part}. Hence, we obtain the desired estimate  \eqref{eq:toshow2}.
\end{proof}

Now, we will obtain suitable estimates for weighted exponential sums using the established bounds for Cesàro exponential sums. For that we will use the following lemmas.

Let us introduce the following notation. \[\Delta_{m,n}:=w(m) - w(n) \ \ \ \ m,n\in \setN.\]
\begin{lem}
Let $W(t) \in \mathbf{E}$ be such that $w(t) = W'(t)$ and $\log t \ll W(t) \ll t$. Then 
\begin{align}
\dfrac{w(t)}{W(t)}\ll \dfrac{1}{t} \label{eq:W1}\\
k\cdot |\Delta_{k,k+1}|\ll w(k) \label{eq:W2}
\end{align}
\end{lem}
\begin{proof}
Since $\displaystyle\lim_{t\to \infty}\dfrac{\log (W(t))}{\log t}=C<\infty$, by L'Hôpital's rule, $\displaystyle\lim_{t\to \infty}\dfrac{ tW'(t)}{ W(t)}=C<\infty$, which establishes \eqref{eq:W1}.

 To see \eqref{eq:W2}, note that $w(t)$ is of the form $w(t)=\dfrac{\phi(t)}{t}$ with $\phi(t)\gg 1$.
 If $\phi(t)\to C$ for some $C\in \setR\setminus \{0\}$, then by  L'Hôpital's rule
 \begin{align*}
C= \lim_{t\to \infty} \dfrac{w(t)}{\frac{1}{t}}=
\lim_{t\to \infty} \dfrac{w'(t)}{-\frac{1}{t^2}}.
 \end{align*}
Hence for sufficiently large $k$,
 \begin{align*}
k\cdot |\Delta_{k,k+1}|\leq  k\cdot \sup_{t\in [k,k+1]} |w'(t)| \leq \dfrac{2|C|k}{k^2}\leq 3 w(k).
 \end{align*} 
 If $\lim_{t\to \infty} \phi(t)=\infty$, then by  L'Hôpital's rule we have
 \begin{align*}
 \lim_{t\to \infty}\dfrac{t w(t)}{W(t)}=\lim_{t\to \infty} \dfrac{t w'(t)+w(t)}{w(t)}
 \end{align*}
 By \eqref{eq:W1}, the limit on the left-hand side is finite, and hence  $t w'(t)\ll w(t)$, as desired.
Since $w(t)$ is eventually monotone, by the mean value theorem, this implies \eqref{eq:W2}.
\end{proof}

\begin{lem}\label{lem:largeregime} Let $a(x)$ and $W(x)$ be as in Proposition~\ref{prop:exponential}.
Then there exist constants $t_1, C>0$ such that for all $t\geq t_1$ and $\beta$ with $\frac{1}{\log^{s+1} t}\leq |\beta| \le \frac{1}{2}$ we have
\begin{equation*}
\left| \hat{\setA}^{\mathsf{a},W}_{t}(\beta)\right|\leq \frac{5C}{\left(\log W(t)\right)^{\frac{1}{2}+\epsilon}}.
\end{equation*}
The constants $C$ and $ t_1$ depend only on $a(x), W(x)$ and $\epsilon$.
\end{lem}
\begin{proof}
We want to show that there exists a constant $C>0$ such that for sufficiently large $t$, in the given range of $\beta$, we have
\begin{equation*}\label{eq:toshow3}
\frac{1}{W(t)}\left|\sum_{n\leq t} w(n)e\left( \lfloor a(n)\rfloor \beta\right) \right|\leq \frac{5C}{\left(\log W(t)\right)^{\frac{1}{2}+\epsilon}}.
\end{equation*}
To establish this, we will use Proposition~\ref{prop:integerpart}. Let $t_0$ be as in Proposition~\ref{prop:integerpart} and $b(n)=\lfloor a(n)\rfloor$. Let $t$ be a large positive number. Assume that $r:=r(t)$ is a positive number (which will be specified later) such that $r$ is much smaller than $t$ but much larger than $t_0$.
By applying Lemma~\ref{lem:byparts} with $g_k= \Delta_{k,t}$, $h_k= \sum_{l=1}^{k-1} e(b(l)\beta)$, $m=r$ and $n=t-1$, we get
\begin{equation*}
\begin{aligned}
\sum_{k=r}^{t-1} \Delta_{k,t} e\bigl(b(k)\beta\bigr)
= -\Delta_{r,t} \sum_{l=1}^{r-1} e\bigl(b(l)\beta\bigr)  - \sum_{k=r}^{t-1} 
   \left(\sum_{l \le k} e\bigl(b(l)\beta\bigr)\right)
   \Delta_{k+1,k}.
\end{aligned}
\end{equation*}
Rewriting the above equation and using the fact that $\Delta_{t,t}=0$, we get
\begin{equation*}
\begin{aligned}
\sum_{k=1}^{t} \Delta_{k,t} e\left(b(k)\beta\right)
&=\sum_{k=r}^{t-1} \Delta_{k,t} e\left(b(k)\beta\right) + \sum_{k=1}^{r-1} \Delta_{k,t} e\left(b(k)\beta\right) \\
&= \Delta_{t,r}\sum_{l=1}^{r-1} e(b(l)\beta)+ \sum_{k=r}^{t-1} \left(\sum_{l\leq k} e(b(l)\beta)\right) \Delta_{k,k+1}+\sum_{k=1}^{r-1} \Delta_{k,t} e\left(b(k)\beta\right)\\
&= \sum_{k=r}^{t-1} \left(\sum_{l\leq k} e(b(l)\beta)\right) \Delta_{k,k+1}+\sum_{k=1}^{r-1} \Delta_{k,r} e\left(b(k)\beta\right)
\end{aligned}
\end{equation*}
Hence, we get
\begin{equation}
\begin{aligned}\label{eq:33}
\frac{1}{W(t)}\sum_{k\leq t} w(k)e\left( b(k) \beta\right) &= \frac{1}{W(t)} \sum_{k\leq {t}} w(t)  e( b(k) \beta)+\frac{1}{W(t)}\sum_{k=r}^{t-1}\Delta_{k,k+1}\left(\sum_{l=1}^{k}  e( b(l) \beta)\right)\\
\quad&\ \ \ \ \ \ \ \ \ \ \ \ \ \ \ \ \ \ \ \ \ \ \ \ \ \ \ \ \ \ \ \ + \frac{1}{W(t)} \sum_{k=1}^{r-1} \Delta_{k,r} e( b(k) \beta)\\
&= (1)+(2)+(3) \text{ (say) .}
\end{aligned}
\end{equation}
Let $\beta\in \setR$ be such that $\frac{1}{r^{d-\frac{1}{2}}}\leq |\beta|\leq \frac{1}{2}$.\\ Then  by Proposition~\ref{prop:integerpart} there are constants $C>0$ and $t_0$ (depending only on $a(x), W(x)$ and $\epsilon$) such that for all $k\geq r\geq t_0$ we have 
\begin{equation}\label{eq:lessthanW}
\left|\frac{1}{k}\sum_{n=1}^k  e( b(n) \beta)\right|\leq \frac{C}{\left(\log W(k)\right)^{\frac{1}{2}+\epsilon}} \leq \frac{C}{\left(\log W(r)\right)^{\frac{1}{2}+\epsilon}},
\end{equation}
Using this along with \eqref{eq:W1}, we get
\begin{equation}\label{eq:34}
|(1)|= \frac{w(t)}{W(t)} \left|\sum_{k\leq {t}}  e( b(k) \beta)\right| \ll \frac{1}{t} \left|\sum_{k\leq {t}}  e( b(k) \beta)\right| \leq    \frac{C}{\left(\log W(r)\right)^{\frac{1}{2}+\epsilon}}
\end{equation}
Similarly, \eqref{eq:W2} and \eqref{eq:lessthanW} give
\begin{equation}
\begin{aligned}\label{eq:35}
|(2)|\leq  \frac{1}{W(t)}\sum_{k=r}^{t-1}k\cdot |\Delta_{k,k+1}|\cdot \left|\frac{1}{k}\sum_{l=1}^k  e( b(l) \beta)\right| &\leq \frac{1}{W(t)}\sum_{k=r}^{t-1} w(k)\cdot \frac{C}{\left(\log W(r)\right)^{\frac{1}{2}+\epsilon}}\\
&\leq \frac{C}{\left(\log W(r)\right)^{\frac{1}{2}+\epsilon}}.
\end{aligned}
\end{equation}

\begin{equation*}\label{eq:36}
|(3)| \leq \frac{1}{W(t)} \sum_{k\leq r-1} w(k)\leq  \frac{ W(r)}{W(t)}.
\end{equation*}

Now, we define $r$ as follows
\[r=r(t):=\max\left\{u\in \setN:  W(u)\leq \left(\dfrac{W(t)}{\left(\log W(t)\right)^{\frac{1}{2}+\epsilon}}\right)\right\}.\]
Define $t_2$ to be a positive number such that $t\geq t_2$ implies $r\geq t_0$.
Then by the above discussion, for any $t\geq t_2$, \eqref{eq:34} and \eqref{eq:35} both hold. Observing that $\log W(r) \geq \frac{1}{2} \log W(t)$ and using \eqref{eq:33}, for any $\beta$ with $\frac{1}{r^{d-\frac{1}{2}}}\leq |\beta|\leq \frac{1}{2}$ we have 
\begin{equation*}
\left|\frac{1}{W(t)}\sum_{k\leq t} w(k)e\left( b(k) \beta\right)\right|\leq \frac{2C}{\left(\log W(t)\right)^{\frac{1}{2}+\epsilon}}. 
\end{equation*}
Now we show that for sufficiently large $t$, $\frac{1}{r^{d-\frac{1}{2}}}\leq \frac{1}{\log^{s+1} t}$. This will follow if we can show that, for any fixed positive number $\eta_1$, $(r+1)\gg \log^ {\eta_1} t$. Since $W$ is an increasing function, we need to show that $W(r+1)\gg W(\log^{\eta_1} t)$. 
By definition of $r$, this will follow if we can show that 
\begin{align*}
W(t)\gg \left(\log W(t)\right)^{\frac{1}{2}+\epsilon}\cdot W(\log ^{\eta_1} t).
\end{align*}
To see this let ${\eta_2}$ be a real number such that ${\eta_2}\in (0,1)$ and ${\eta_1}{\eta_2}<\frac{1}{4}$.\\
If $\log t \ll W(t)\ll t^{\eta_2}$, then we have
\begin{align*}
W(t)\gg \log t\gg (\log t)^{\frac{1}{2}+\epsilon}\cdot \log^{{\eta_1}{\eta_2}} t \gg \left(\log W(t)\right)^{\frac{1}{2}+\epsilon}\cdot W(\log ^{\eta_1} t).
\end{align*}
If $t^{\eta_2} \ll W(t)\ll t$, then we have
\begin{align*}
W(t)\gg t^{\eta_2} \gg (\log t)^{\frac{1}{2}+\epsilon}\cdot \log^{{\eta_1}} t \gg \left(\log W(t)\right)^{\frac{1}{2}+\epsilon}\cdot W(\log ^{\eta_1} t).
\end{align*}
Let $t_3$ be such that $t\geq t_3$ implies $\frac{1}{r^{d-\frac{1}{2}}}\leq \dfrac{1}{\log^{s+1} t}$.
Finally, taking $t_1:=\max\{t_2,t_3\}$, the conclusion of the lemma follows.
\end{proof}

\begin{proof}[Proof of Proposition~\ref{prop:exponential}]
The proof now follows from Lemma~\ref{lem:smallregime} and Lemma~\ref{lem:largeregime}. 
\end{proof}

\section{Ces\`aro averages along subpolynomials of degree $1$}\label{deg:1}
This section is dedicated to proving the following theorem.
\begin{thm}[Theorem \ref{thm:deg1:int}] \label{thm:deg1}
Let $a(x)\in \mathbf{U}$ have the canonical decomposition $a(x)=cx+r(x)$ with $c\neq 0$.
    If either $r(x) \succ \log^ {2+\epsilon} x$ for some $\epsilon>0$, or $c=\frac{1}{m}$ for some $m\in \setZ$,  then $(\lfloor a(n) \rfloor)$ is pointwise $L^2$-good.
\end{thm}

\begin{rem}
\begin{enumerate}[(i)]\item If $r(x)=0$, then $(a(n))=(cn)$ is pointwise $L^1$-good for any measure preserving flow, and hence, by \cite[Theorem 1]{BJW},  $(\lfloor a(n) \rfloor)$ is pointwise $L^2$-good.
    \item The case $c=0$ is taken care of in Theorem~\ref{thm:sufficient}.
    \item If $1\prec r(x)\prec \log x$ and $1/c \notin \mathbb{Z}$, then by \cite[Theorem 3.2]{BKQW}, $\lfloor a(n) \rfloor$ is bad for mean convergence, and hence it is pointwise $L^2$-bad. 
\end{enumerate}
\end{rem}

To prove this theorem,  we require some modifications of \cite[Theorem 7.1]{BKQW} because this theorem does not apply to the degree $1$ case. Let $F(x)$ be the compositional inverse of $a(x)$.
We will compare 
$$\setA^{\mathsf{a}}_tf(x):= \displaystyle\frac{1}{t} \sum_{n\leq t} f(T^{\lfloor a(n) \rfloor}x) \text{ with }\setB_tf(x):= \displaystyle\frac{1}{t} \sum_{n\leq a(t)} F'(n) f(T^n x).$$
Note that $\lim_{t \rightarrow \infty} \setB_t f(x)$ exists for almost every $x$. 
This can be justified as follows:\\
If $F'(n)$ is non-increasing, then Lemma~\ref{lem:HW-2} (a) gives the desired result. 
If  $F'(n)$ is non-decreasing, then
\[ \lim_{N \rightarrow \infty}  \dfrac{N F'(N)}{\sum_{n=1}^N F'(n)} = \lim_{y \rightarrow \infty} \frac{y  F'(y)}{F(y)} =  \lim_{y \rightarrow \infty} \frac{\log F(y)}{\log y} = \lim_{x \rightarrow \infty} \frac{\log F(a(x))}{\log a(x)} = \lim_{x \rightarrow \infty} \frac{\log x}{\log a(x)} =1. \]
Hence,
\begin{align*}
\sup_N \dfrac{N F'(N)}{\sum_{n=1}^N F'(n)} <\infty.
\end{align*} Then, invoking Lemma~\ref{lem:HW-2} (b), we obtain the desired conclusion.

\begin{rem}
This is one instance where extra care is needed in the degree one case, unlike in the case of degree two or higher, where $F'(n)$ always decreases to $0$. See the first line of page 88 in \cite{BKQW}.
\end{rem}

As in the proof of Theorem~\ref{thm:sparse2}, we consider the Fourier transforms of $\setA^{\mathsf{a}}_t f$ and $\setB_tf$:
$$\displaystyle\hat{\setA}^{\mathsf{a}}_t (\beta) := \dfrac{1}{t}\sum_{n\leq t} e( \lfloor a(n)\rfloor \beta ) \quad \text{and} 
\quad \displaystyle\hat{\setB}_t (\beta) := \dfrac{1}{t}\sum_{n\leq a(t)} F'(n) e( n \beta).$$
To prove this theorem, we will use the following proposition, the proof of which is given at the end of this section. 

\begin{prop}{}\label{prop:expsum:deg1}
Let $a(x)\in \mathbf{U}$ have the canonical decomposition $a(x)=cx+r(x)$ with $c\neq 0$. Suppose that  $r(x) \succ \log^ {2+\epsilon} x$. 
Then there exists $\tau>0$ such that for $0 \leq  |\beta| \leq {1/2}$
\begin{align*}
\left|\dfrac{1}{t}\sum_{n\leq t} e( \lfloor a(n)\rfloor \beta )-\dfrac{1}{t}\sum_{n\leq a(t)} F'(n) e( n \beta)\right|
\ll  \dfrac{1}{(\log t)^{1/2+\tau}}.
\end{align*}
\end{prop}

\begin{proof}[Proof of Theorem~\ref{thm:deg1}]
Without loss of generality, we assume that $c>0$. (The proof for $c< 0$ is analogous and omitted.)

\noindent\textbf{Case (I).} Let $r(x)\succ (\log x)^{2+\epsilon}$ for some $\epsilon>0$.\\
Following the argument of Theorem~\ref{thm:sparse2}, Theorem~\ref{thm:deg1} will follow from Proposition \ref{prop:expsum:deg1}. 

\noindent\textbf{Case (II).} Suppose that $c=\frac{1}{m}$ for some $m \in \mathbb{N}$. \\
By Case (I), it suffices to consider the case $r(x) \prec (\log x)^{4}$.
Clearly, it will be sufficient to show that for each $l= 0, 1, \dots, m-1$ and for any $f\in L^1$,
\[\frac{1}{N} \sum_{n\in [N] } f(T^{\lfloor a(mn +l) \rfloor}x) \text{ converges a.e. as } N\to \infty. \]
Consider the set $A:=\left\{\lfloor a(mn+l) \rfloor: n\in \setN\right\}$. Write
\begin{equation*}
\dfrac{1}{N}\sum_{n\in [N]\cap A} f(T^n x)=\dfrac{1}{N}\sum_{n\in[N]} f(T^n x) -\dfrac{1}{N}\sum_{n\in [N]\cap A^c} f(T^n x).
\end{equation*}
By Birkhoff's pointwise ergodic theorem, $\displaystyle\lim_{N\to\infty}\dfrac{1}{N}\sum_{n\in[N]} f(T^n x)$ exists for a.e. $x$. 
Hence, it will be sufficient to show that
$\displaystyle\lim_{N\to \infty}\dfrac{1}{N}\sum_{n\in [N]\cap A^c} f(T^n x)= 0$ for a.e. $x$.
Observe that \[\dfrac{ \left| \left\{\lfloor a(mn+l) \rfloor: n\leq N\right\}\triangle \{1,2,\cdots, N\}\right|}{N} \leq \dfrac{2(\log \left((m+1)N)\right)^4}{N}\ll \dfrac{1}{N^{1/2}},\]
where $\triangle$ means symmetric difference. 
So, the proof now follows from the following lemma.

\begin{lem}
Let $(a_n)$ be a sequence of positive integers. Then for any $\epsilon>0$, any measure preserving system $(X,\mathcal{B},\mu,T)$, $f \in L^1(\mu)$ and for almost every $x\in X$
\begin{equation*}
\lim_{N \rightarrow \infty} \dfrac{1}{N^{1+\epsilon}}\sum_{n=1}^N f(T^{a_n} x)=0.
\end{equation*}
\end{lem}
\begin{proof}
Let $f\in L^1(\mu)$. Without loss of generality, we can assume that $f\geq 0$. For a positive integer $N$, let us define
\begin{equation*}
T_N f(x)=\dfrac{1}{N^{1+\epsilon}}\sum_{n=1}^N f(T^{a_n} x).
\end{equation*}
 Let $\tau>0$ be arbitrary. Consider the set 
\begin{equation*}
E_n:=\left\{x: T_n f(x)>\tau\right\}. 
\end{equation*}
By the Markov's inequality, for every $\sigma>1$ we have
\begin{equation*}
\sum_{n=1}^\infty \mu\left(E_{\lfloor \sigma^n\rfloor}\right)\leq \sum_{n=1}^\infty \dfrac{||f||_{L^1}}{\tau (\sigma^{n}-1)^\epsilon}<\infty.
\end{equation*}
Hence, by the Borel-Cantelli lemma, for almost every $x\in X$, we have 
$$\limsup_{n\to \infty}T_{\lfloor \sigma^n \rfloor} f(x)\leq \tau.$$ 
Since $\tau>0$ is arbitrary, for almost every $x\in X$, we have  
$$\limsup_{n\to \infty}T_{\lfloor \sigma^n \rfloor} f(x)=0.$$
 Finally, applying an argument similar to that in Lemma~\ref{lem:reduction2}, we obtain the desired conclusion.
\end{proof} \end{proof} 

\begin{rem}
Let $1\leq p<\infty$. If a set $B\subset \setN$ has zero density, it does not necessarily imply that for every $f\in L^p$, $\frac{1}{N}\sum_{n\in [N]\cap B} f(T^n x)\to 0$ for a.e. $x$. See \cite[Theorem B]{QW} for such a counterexample.
\end{rem}

It remains to prove Proposition \ref{prop:expsum:deg1}. 
We need several lemmas for proving this proposition. A proof of the following two lemmas can be found in \cite[Chapter 2]{GK}. 
\begin{lem}[Kusmin-Landau]\label{Lemma:KL}
If $f$ is continuously differentiable, $f'$ is monotonic, and $||f'||\geq \kappa>0$ on an interval $I$. Then there exists an absolute constant $C>0$ such that
\begin{align*}
\left| \sum_{n\in I} e(f(n)) \right| \leq \frac{C}{\kappa}.
\end{align*}
\end{lem}

The following lemma is a special case of Lemma~\ref{lem:van der Corput}. 
\begin{lem}[Van der Corput]\label{lem:vdc2}
Suppose that $f$ is a real valued function that is twice continuously differentiable on an interval  $I$. Suppose also that there is some $\kappa>0$ and some $A\geq 1$ such that 
\begin{align*}
\kappa \leq |f''(x)|\leq A \kappa,
\end{align*}
on $I$, then there exists an absolute constant $C>0$ such that
\begin{equation*}
\left| \sum_{n\in I} e(f(n)) \right| \leq C \left(A |I| \kappa^{1/2}+\kappa^{-1/2}\right).
\end{equation*}
\end{lem}

\begin{lem}[Proposition 2.2 in \cite{BKS}]
\label{lem:est-lopital-slow1}
If $r(x) \in \mathbf{U}$ satisfies $\log x \prec r(x) \prec x$, then \[ \left| \frac{x r'(x)}{r(x)} \right| \ll 1 \quad \text{and} \quad \left|\frac{x r''(x)}{r'(x)} \right| \ll 1.\]
In particular, $|r''(x)| \ll \frac{|r(x)|}{x^2}$.
\end{lem}

Applying L'Hôpital's rule as in Lemma \ref{lem:Lopital}, we also have the following estimates.
\begin{lem}
\label{lem:est-lopital-slow2}
Let $r(x) \in \mathbf{U}$. 
\begin{enumerate}[(a)]
\item If $\log^a x \ll |r(x)| \prec x$ for some $a >0$, then $|r''(x)| \gg \frac{\log^{a-1} x}{x^2}$.
\item If $ x^b \ll |r(x)| \prec x$ for some $b>0$, then $|r''(x)| \gg x^{b-2}$.
\item If  $|r(x)| \ll x^b$ for some $0< b <1$, then $|r''(x)| \ll x^{b-2}$.
\end{enumerate}
\end{lem}

\begin{lem}\label{lem:6}
Let $a(x)\in \mathbf{U}$ have the canonical decomposition $a(x)=cx+r(x)$ with $c\neq 0$. Suppose that  $r(x) \succ \log^ {2+\epsilon} x$.
If $0\le|\beta|\leq \dfrac{1}{t^{1/2}}$, then 
\begin{align*}
\left|\dfrac{1}{t}\sum_{n\leq t} e( \lfloor a(n)\rfloor \beta )-\dfrac{1}{t}\sum_{n\leq a(t)} F'(n) e( n \beta)\right|
\ll  \dfrac{1}{\log t}.
\end{align*}
\end{lem}

\begin{proof}  
Observe that in the given range of $\beta$
$$\left| e( \lfloor a(n)\rfloor \beta )- e(  a(n) \beta )\right| \leq 2\pi |\beta|\ll \frac{1}{t^{1/2}}\leq \frac{1}{\log t}.$$ 
{ Thus, 
$$\left| \frac{1}{t} \sum_{n \leq t} e( \lfloor a(n)\rfloor \beta )-  \frac{1}{t} \sum_{n \leq t} e(  a(n) \beta )\right| \ll \frac{1}{\log t}.$$ 
}
Hence, it will be sufficient to show that
\begin{align}
\label{eq:8.11}
\left|\dfrac{1}{t}\sum_{n\leq t} e(  a(n) \beta )-\dfrac{1}{t}\sum_{n\leq a(t)} F'(n) e( n \beta)\right|
\ll  \dfrac{1}{\log t}.
\end{align}

Since $F(x)$ is the compositional inverse of $a(x)$, by change of variable we have
\begin{equation}\label{eq:8.12}
\int_{0}^t e( a(x)\beta)dx= \int_{a(0)}^{a(t)} F'(x) e( x\beta) dx.
\end{equation}
For any $x\in [n,n+1]$,
$$|F'(x) e(x\beta)-F'(n) e(n \beta)|\leq |F'(x)- F'(n)|+|F'(n)||e( x \beta)- e( n\beta)|.$$ 
So, for some $C>0$, we have
\begin{equation}\label{eq:8.13}
\begin{aligned}
\left| \int_{a(0)}^{a(t)} F'(x)\,e(x\beta)\,dx 
- \sum_{n\le a(t)} F'(n)\,e(n\beta) \right|
&\le C +\sum_{n\le a(t)} |F'(n+1)-F'(n)|
    + C|\beta| \sum_{n\le a(t)} |F'(n)| \\
&\le C + |F'(1)|+|F'(a(t))|+ C|\beta|\,F(a(t))\\
&\ll  (1 + |\beta|t),
\end{aligned} 
\end{equation}
where in the last line we used the fact that $\lim_{x\to\infty}\frac{a(x)}{x}$ is finite, so $\lim_{x\to\infty}\frac{F(x)}{x}$ is also finite, and hence $F'$ is bounded.\\
Similarly, for any $x\in [n,n+1]$, $|e(a(x)\beta)-e(a(n)\beta)|\leq 2\pi |a'(n+1)| |\beta|$. 
So, 
\begin{align}\label{eq:8.14}
\left|\int _{0}^t e(a(x)\beta) dx- \sum_{n\leq t}e(a(n)\beta)\right|\leq C a(t+1) |\beta| \ll  t |\beta|.
\end{align}
Combining \eqref{eq:8.12}, \eqref{eq:8.13} and \eqref{eq:8.14}, \eqref{eq:8.11} follows.
\end{proof}

\begin{lem}\label{lem:7} Let $a(x)\in \mathbf{U}$ have the canonical decomposition $a(x)=cx+r(x)$ with $c\neq 0$. Suppose that  $r(x) \succ \log^ {2+\epsilon} x$.
Then there exists $\tau>0$ such that for every $\beta$ with $\dfrac{1}{t^{1/2}}\leq |\beta|\leq \frac{1}{2}$, we have
\begin{align} &\left|\dfrac{1}{t}\sum_{n\leq a(t)} F'(n)e(n \beta)\right|\leq \dfrac{C}{(\log t)^{1/2+\tau}}\label{eq:degree1bound2} \text{ and }\\
&\left|\dfrac{1}{t}\sum_{n\leq t} e(\lfloor a(n) \rfloor \beta)\right| \leq \dfrac{C}{(\log t)^{1/2+\tau}}\label{eq:deg1bound1}.
\end{align}
\end{lem}

\begin{proof} To establish \eqref{eq:degree1bound2}, by summation by parts, we get

\begin{align}
\left|\sum_{n\leq a(t)} F'(n)e(n \beta)\right|&= \left|\sum_{k=1}^{a(t)} (F'(k)-F'(k-1))\cdot \sum_{n=k}^{a(t)} e(n\beta) \right|\notag\\
&\leq \left| \sum_{k=1}^{a(t)} (F'(k)-F'(k-1))\cdot \dfrac{1}{|e(\beta)-1|}\right| \notag\\
&\leq {F'(a(t))}\cdot {t^{1/2}} \label{eq:gp}\\&\leq {C} {t^{1/2}} \label{eq:gp1}.
\end{align}
To obtain \eqref{eq:gp}, we used $\frac{1}{t^{1/2}}\leq |\beta| \leq \frac{1}{2}$ so $|e(\beta)-1|\geq \frac{1}{t^{1/2}}$, and to obtain \eqref{eq:gp1}, we used $F(t)\ll t$, so $F'(t)\ll 1$.
\eqref{eq:degree1bound2} now follows by dividing each expression by $t$.

By Lemma~\ref{lem:integer part}, \eqref{eq:deg1bound1} will follow if we can show the following:\\ There exists $\tau>0$ such that for every $\beta$ with $\dfrac{1}{t^{1/2}}\leq |\beta|\leq (\log t)^3$,
\begin{align}\label{eq:tau}
\left|\dfrac{1}{t}\sum_{n\leq t} e( a(n)  \beta)\right|  \leq \dfrac{C}{(\log t)^{1+\tau}}.
\end{align}

To establish this, we will consider two subcases.\\
\textbf{Subcase (a).} Let $\dfrac{1}{t^{1/2}}\leq |\beta|\leq \dfrac{1}{\log^{\gamma} t}$ for $\gamma\in (0,1)$ to be chosen later.\\
Let $f(x)= a(x)\beta$. Then  in the interval $[t^{7/8}, t]$, $||f'(x)||\geq \frac{c}{2t^{1/2}}$ for some $c>0$. Hence, by Lemma~\ref{Lemma:KL}, there exists $C>0$ such that
\begin{equation*}
\left|\dfrac{1}{t}\sum_{n\leq t} e( a(n) \beta)\right|
 \leq \left|\dfrac{1}{t}\sum_{n\leq t^{7/8}} e( a(n)  \beta)\right| + \left|\dfrac{1}{t}\sum_{t^{7/8} < n\leq t} e( a(n)  \beta)\right|
\leq C\left(\frac{t^{7/8}}{t} + \dfrac{1}{t^{1/2}}\right).
\end{equation*}

\noindent\textbf{Subcase (b).} Let $\dfrac{1}{\log^{\gamma} t}\le |\beta|\leq (\log t)^3$.\\
First, suppose that $\log ^{\eta} x\prec r(x)\ll x^{1/2}$ for $\eta > 2$.
Let $f(x)= a(x)\beta$. Then $|f''(x)|=|r''(x)\beta|$. On the interval $[t^{7/8}, t]$, by Lemma \ref{lem:est-lopital-slow2},  we have 
\[\dfrac{(\log t)^{\eta-1} |\beta|}{t^2}\ll |r''(x)\beta|= |f''(x)|\leq |r''(t^{7/8})\beta| \ll \dfrac{|\beta|}{t^{\frac{21}{16}}}. \]
So we can take $\kappa= \dfrac{(\log t)^{\eta-1} |\beta|}{t^2}$ and $A=t^{11/16}$ and use Lemma~\ref{lem:vdc2} to obtain
\begin{align*}
\left|\dfrac{1}{t}\sum_{n\leq t} e( a(n) \beta)\right|&=\left|\dfrac{1}{t}\sum_{n\leq t^{7/8}} e( a(n) \beta)\right|+\left|\dfrac{1}{t}\sum_{t^{7/8}\leq n\leq t} e( a(n) \beta)\right|\\ &\ll \dfrac{t^{7/8}}{t}+ \dfrac{t^{11/16}(\log t)^{(\eta-1)/2}|\beta|^{1/2}}{t}+\dfrac{1}{(\log t)^{\frac{(\eta-1)}{2}}|\beta|^{1/2}}
\\ &\ll \dfrac{1}{(\log t)^{\frac{(\eta-1-\gamma)}{2}}}.
\end{align*}
Since $\eta>2$, we can choose $\gamma>0$ small enough so that \eqref{eq:tau} is satisfied for some $\tau>0$.

Now, let $x^{1/2}\prec  r(x)\prec x$.
Let $f(x)= a(x)\beta$, so that $|f''(x)|=|r''(x)\beta|$. Then, by Lemma \ref{lem:est-lopital-slow1} and Lemma \ref{lem:est-lopital-slow2}, we have
\begin{align*}
\label{eq:question}
\dfrac{ x^{1/2} |\beta|}{x^2}\ll |r''(x)\beta|= |f''(x)|\ll \dfrac{ |r(x)\beta|}{x^{2}}\ll \dfrac{ |x\beta|}{x^{2}}.
\end{align*}
Hence, in the interval $[t^{\frac{7}{8}}, t]$, we have 
\[\dfrac{ t^{1/2} |\beta|}{t^2}\ll |r''(x)\beta|= |f''(x)| \ll \dfrac{ t^{\frac{7}{8}}|\beta|}{t^{\frac{7}{4}}}.\]
Applying Lemma~\ref{lem:vdc2}, with $\kappa= \dfrac{t^{1/2} |\beta|}{t^2}$ and $A=t^{5/8}$,  we get
\begin{align*}
\left|\sum_{{1}\leq n\leq t} e( a(n) \beta)\right|&\leq t^{7/8}+\left|\sum_{t^{7/8}\leq n\leq t} e( a(n) \beta)\right| \\
&\ll t^{7/8}+ \dfrac{(t-t^{7/8})t^{5/8}\cdot t^{1/4}|\beta|^{1/2}}{t}+\dfrac{t}{ t^{1/4}|\beta|^{1/2}}.
\end{align*}
\eqref{eq:tau} now follows.
\end{proof}
\begin{proof}[Proof of \Cref{prop:expsum:deg1}]
The proof of this proposition follows immediately from \Cref{lem:6} and \Cref{lem:7}.
\end{proof}

\section{Log averages along subpolynomials}
\label{sec:sparse2} 
In this section, we prove the following result. 
\begin{thm}[Theorem \ref{sparse2':intro}]
\label{thm:sparse2'}
Suppose that $a(x) \in \mathbf{U}$ has the canonical decomposition $a(x) = p(x) + r(x)$, where $|r(x)|\gg (\log x)^\eta$, for some $\eta>1$. Then for any system $(X,\mathcal{B},\mu,T)$ and any $f\in L^2(\mu)$, 
\begin{equation}
\label{eq5.1}
\lim_{N\to \infty}\frac{1}{\log N} \sum_{n\leq N} \frac{1}{n} f(T^{\lfloor a(n) \rfloor}x) \text{ exists for almost every} \ x\in X.
\end{equation}
Moreover, if the system is ergodic then the limit in \eqref{eq5.1} is $\int_X f d\mu.$
\end{thm}

\begin{rem}
\label{rem:thm:sparse2}
Note that the case $\deg(p)=0$ in Theorem \ref{thm:sparse2'} follows from Theorem B in \cite{BW1996} together with Theorem \ref{thm:sufficient}. Indeed, when $\deg(p)=0$, if $r(x)\succ x^\epsilon$ for some $\epsilon>0$, then $(\lfloor a(n)\rfloor)$ is pointwise $L^p$-good for Ces\`aro averages by Theorem B in \cite{BW1996}. On the other hand, if
\[
(\log x)^\eta \ll |r(x)| \prec x,
\]
then $(\lfloor a(n)\rfloor)$ is pointwise $L^1$-good for logarithmic averages by Theorem \ref{thm:sufficient}.
Note also that  the case $\deg(p) \geq 2$ in Theorem \ref{thm:sparse2'} follows from Corollary~\ref{cor:logcase}.
\end{rem}

By Remark \ref{rem:thm:sparse2} and Theorem \ref{thm:deg1},  it is sufficient to prove Theorem \ref{thm:sparse2'} in the case $\deg(p) = 1 $ and $|r(x)| \ll \log^3 x$.
 Hence, from now on we assume that $\deg(p) =1$  and $|r(x)| \ll \log^3 x$.
The proof of this theorem follows a strategy similar to that used in the proof of Theorem~\ref{thm:sparse2}, so we will skip some of the details.  
Let us fix a measure preserving system $(X,\mathcal{B},\mu,T)$ and a function $f\in L^2(\mu)$. For a real-valued function $a(x)$ and a positive real number $t$, we use the following logarithmic-averaging notation.
\begin{align*}
&\setA^{\mathsf{a}}_tf(x):= \frac{1}{t} \sum_{n\leq t} f(T^{\lfloor a(n) \rfloor}x),\ \  \setA^{\mathsf{a},L}_tf(x):= \frac{1}{\log t} \sum_{n\leq t} \frac{1}{n} f(T^{\lfloor a(n) \rfloor}x),\\
&\setB^{\mathsf{a}}_tf(x):= \frac{1}{t} \sum_{n\leq t} f(T^{\lfloor p(n)+\log^3 n\rfloor}x),\ \  \setB^{\mathsf{a},L}_tf(x):= \frac{1}{\log t} \sum_{n\leq t} \frac{1}{n} f(T^{\lfloor p(n)+\log^3 n\rfloor}x).
\end{align*}
 As before, the Fourier transform of the operator $\setA^{\mathsf{a}}_t$ is given by  $\displaystyle\hat{\setA}^{\mathsf{a}}_t (\beta)=\frac{1}{t}\sum_{n\leq t} e\left( \lfloor a(n) \rfloor\beta\right)$, and the Fourier transforms of the operators $\setA^{\mathsf{a},L}_t, \setB^{\mathsf{a}}_t$ and $\setB^{\mathsf{a},L}_t$ are given by analogous expressions.

By Theorem \ref{thm:deg1}, $\lim_{t \rightarrow \infty} \setB^{\mathsf{a}}_t f(x)$ exists for almost every $x$ and the limit is $\int_X f \, d \mu$ if the system is ergodic. Hence, by Lemma~\ref{lem:comp:weights},  $\lim_{t \rightarrow \infty} \setB^{\mathsf{a},L}_tf(x)$ exists for almost every $x$ and the limit is $\int_X f \, d \mu$ if the system is ergodic.

Following the argument in the proof of Theorem~\ref{thm:sparse2}, it is enough to show the following result.

\begin{prop}{}\label{prop:exponential'}
Suppose $a(x)$ has the canonical decomposition $a(x) = p(x) + r(x)$, where $\deg (p) = 1$ and $\log^3 x\gg|r(x)|\gg (\log x)^\eta$ for some $\eta>1$.
Then there exist positive constants $t_0$ and $C$ such that for all $t\geq t_0$ and  for all  $\beta$ with $0\leq  |\beta|\leq \frac{1}{2}$, we have
\begin{equation}\label{eq:toshow22'}
\left| \hat{\setA}^{\mathsf{a},L}_{t}(\beta)-\hat{\setB}^{\mathsf{a},L}_{t}(\beta)\right|\leq \frac{C}{\log\log t}.
\end{equation}
\end{prop}

The proposition is proved in two parts: 
for small values of $\beta$, the desired estimate, i.e. \eqref{eq:toshow22'}, is established in Lemma~\ref{lem:smallregime'}, whereas for large values of $\beta$, we prove the advertised estimate in Lemma~\ref{lem:largeregime'}.

\begin{lem}\label{lem:smallregime'}
Let $a(x)$ be as in Proposition~\ref{prop:exponential'}. Then there exist constants $t_0, C>0$ such that for all $t\geq t_0$ and $\beta$ with $0\leq |\beta|\leq \dfrac{1}{\log^{5} t}$, we have
\begin{equation*}
\left| \hat{\setA}^{\mathsf{a},L}_{t}(\beta)-\hat{\setB}^{\mathsf{a},L}_{t}(\beta)\right| \le \frac{C}{\log\log t}.
\end{equation*}
\end{lem}

\begin{proof}
Suppose $\beta$ satisfies the hypothesis of the lemma. Then
\begin{align}
\left| \hat{\setA}^{\mathsf{a},L}_{t}(\beta)-\hat{\setB}^{\mathsf{a},L}_{t}(\beta)\right|&= \left|\frac{1}{\log t}\sum_{n\leq t} \frac{1}{n} \left( e\left( \lfloor a(n) \rfloor\beta\right)-e\left( \lfloor p(n)+\log^3 n\rfloor\beta\right) \right)\right|\notag\\
&\leq \frac{1}{\log t}\sum_{n\leq t} \frac{1}{n} \left|e\left( \lfloor a(n) \rfloor\beta\right)-e\left( \lfloor p(n)+\log^3 n\rfloor\beta\right)\right|\notag\\
&\leq \frac{4\pi |\beta|}{\log t}\sum_{n\leq t} \frac{\log^3 n}{n} \label{eq:31'}\\
&\ll \frac{8\pi |\beta|}{\log t} \cdot {\log^{4} t}\label{eq:32'}\\
&\ll \frac{1}{\log\log t}\notag.
\end{align}
We used the inequality $|e(x)-e(y)|\leq 2\pi |x-y|$ and the fact that $\eta\leq 3$ to obtain \eqref{eq:31'}, and $\displaystyle\sum_{n\leq t}\frac{1}{n}\leq 2 \log t$ for large $t$ to obtain \eqref{eq:32'}.
\end{proof}

To establish the estimate for large values of $\beta$, we will use Lemma \ref{lem:Lopital} and Lemma \ref{lem:van der Corput}.

\begin{lem}\label{lem:largeregimereal'} 
Let $a(x)$  be as in Proposition~\ref{prop:exponential'} and $\epsilon\in (0,\frac{1}{2})$. 
Then for any $\beta$ with ${t^{-\frac{1}{2}}}\leq |\beta| \leq t^\epsilon$, we have
\begin{equation*}\label{eq:toshow1'}
\frac{1}{t}\left|\sum_{n\leq t} e\left( a(n)\beta\right)\right| \ll  \frac{1}{(\log\log t)^3}.
\end{equation*}
The implicit constant here depends only on $a(x)$ and $\epsilon$.
\end{lem}

\begin{proof}
\textbf{Case (a).} Let $\frac{1}{\sqrt{t}}\leq |\beta|\leq \frac{1}{\log \log t}$.\\
Suppose $f(x)=(cx+r(x))\beta$. So, we have $f'(x)$ is monotone and for sufficiently large $x$, $||f'(x)||>\frac{|c\beta|}{2}\geq \frac{|c|}{2\sqrt t}$. Hence, by Lemma~\ref{Lemma:KL},
\begin{equation*}
\label{eq:case1':sec6} 
\left|\frac{1}{t}\sum_{n=1}^t e\left(a(n)\beta\right)\right|\ll \dfrac{1}{\sqrt{t}}\ll \dfrac{1}{(\log\log t)^3}.
\end{equation*}
\textbf{Case (b).} Let $\frac{1}{\log \log t}\leq |\beta|\leq t^\epsilon$.\\
For this case, we want to apply Lemma~\ref{lem:van der Corput} on the function $\phi(x)= a(x)\beta$ with $l=2, Y=t^{1-\epsilon}, X=t$. Since $|r(x)|\gg \log^\eta x$, $r(x)\succ \log^{\eta'}x$ for some $\eta'>1$ (for simplicity we write $\eta'=\eta$). Hence by Lemma~\ref{lem:Lopital}, we have
\begin{equation*}
|\phi^{(2)}(x)|= |r^{(2)}(x) \beta| \gg \left(\dfrac{(\log x)^{\eta-1}}{x^{2}}\right)|\beta|.
\end{equation*}
So, we can take $\rho= \dfrac{(\log t)^{\eta-1}}{t^{2}}|\beta|$. Again using the mean value theorem, Lemma~\ref{lem:Lopital} and the fact that $|r(x)|\ll \log^4 x$, for large $t$,  we have $R\leq \dfrac{C (\log t)^{3}}{t^{2(1-\epsilon)}}|\beta|$ for some $C > 0$. 
Thus, for large $t$, we have
\begin{equation*}
\frac{R^2}{\rho}\ll \frac{ (\log t)^{7-\eta}}{t^{2(1-2\epsilon)}}|\beta|, \frac{1}{\rho(X-Y)^2}\ll \frac{1}{|\beta| (\log t)^{\eta-1}}\text{ and }\frac{R}{\rho(X-Y)}\ll \frac{ (\log t)^{4-\eta}}{t^{1-2\epsilon}}.
\end{equation*}
Since $\eta-1>0$, for large $t$ we have $\frac{1}{|\beta| (\log t)^{\eta-1}}\ll \frac{1}{(\log\log t)^{6}}$.
Hence, by Lemma~\ref{lem:van der Corput}, in the given range we have
\begin{equation*}
\left|\sum_{t^{1-\epsilon}\leq n\leq t} e\left( a(n)\beta\right)\right| \ll  (t-t^{(1-\epsilon)}) \frac{21}{(\log \log t)^3}.
\end{equation*}
Thus,\begin{equation*}\label{eq:case2':sec6}
\frac{1}{t}\left|\sum_{n\leq t} e\left( a(n)\beta\right)\right| \ll \frac{1}{(\log \log t)^3}.
\end{equation*}
Combining Case (a) and Case (b), the conclusion of the lemma follows.
\end{proof}

Now applying Lemma~\ref{lem:integer part} with $\phi(n)= a(n)$, $Q= t^{\frac{1}{2}}$, $r= (\log \log t)^3$ and $S= t^{\epsilon}$, we have the following result.
\begin{prop}\label{prop:integerpart'} Let $a(x)$ and $r(x)$ be as in Proposition~\ref{prop:exponential'}.
Then there exist constants $t_0, C>0$ such that for all $t\geq t_0$ and $\beta$ with $t^{-\frac{1}{2}} \leq |\beta| \leq \frac{1}{2}$, we have
\begin{equation*}\label{eq:toshow2'}
\frac{1}{t}\left|\sum_{n\leq t} e\left( \lfloor a(n)\rfloor \beta\right) \right|\ll \frac{C}{\log\log t}. 
\end{equation*}
\end{prop}

Then, we use the same argument as in the proof of Lemma~\ref{lem:largeregime} with $W(t) = \log t$ to derive the following lemma.
\begin{lem}\label{lem:largeregime'} 
Let $a(x)$ and $r(x)$ be as in Proposition~\ref{prop:exponential'}.
Then for every $\beta$ satisfying $ \frac{1}{\log^{5} t}\leq |\beta| \le\frac{1}{2}$, we have
\begin{equation*}
\left| \hat{\setA}^{\mathsf{a},L}_{t}(\beta)\right|\ll \frac{1}{ \log\log t}.
\end{equation*}
\end{lem}

\begin{proof}[Proof of Proposition~\ref{prop:exponential'}]
The proof now follows from Lemma~\ref{lem:smallregime'} and Lemma~\ref{lem:largeregime'}. 
\end{proof}

\section{Applications to joint ergodicity}
\label{sec:JE}
In this section, we prove Proposition \ref{thm:je1:int}, Theorem \ref{thm:multiconv:int} and Theorem \ref{thm:multiconv2:int}.

\begin{prop}[Proposition \ref{thm:je1:int}]
\label{thm:je1}
Let $a(t) \in \mathbf{U}$ be such that $1 \prec a(t) \prec t$. 
Let $W(t) \in \mathbf{E}$ be such that $w(t):= W'(t)$ is positive and non-increasing and $\lim\limits_{t \rightarrow \infty} W(t) = \infty$.
Suppose that there exists $c > 0$ such that for all large enough $t$,  
\begin{equation*}
W(t) \leq |a(t)|^c. 
\end{equation*} 
If $T_1, \dots, T_k$ are pointwise jointly ergodic, then for any bounded measurable functions $f_1, \dots, f_k$ on $X$, 
\begin{equation}
\label{eq:thm:je1}
\lim_{N \rightarrow \infty} \frac{1}{W(N)} \sum_{n=1}^N w(n) \prod_{i=1}^k f_i (T_i^{\lfloor a(n) \rfloor} x) = \prod_{i=1}^k \int f_i \, d \mu_i \quad \text{for a.e. } x.
\end{equation}
\end{prop}
\begin{proof}
Suppose first that $a(t)$ is eventually nonnegative. By Theorem \ref{thm:slow-weight}, \eqref{eq:thm:je1} holds for any nonnegative bounded measurable functions $f_1, \dots, f_k$. By linearity, it follows that \eqref{eq:thm:je1} holds for arbitrary bounded measurable functions $f_1, \dots, f_k$.

If $a(t)$ is eventually nonpositive, then
$$
T^{\lfloor a(n) \rfloor}=(T^{-1})^{\lceil -a(n) \rceil}.
$$
Applying the preceding argument to $T^{-1}$ and $-a(t)$, we obtain the desired result.
\end{proof}

Before presenting the proofs of Theorem \ref{thm:multiconv:int} and Theorem \ref{thm:multiconv2:int}, we briefly review some basic notions and results on weighted uniform distribution in $\mathbb{R}^k$.  Throughout this section, $\mathbb{T}^k := \mathbb{R}^k/\mathbb{Z}^k$ is equipped with the Lebesgue measure $\lambda^k$.

Let $(w(n))_{n \in \mathbb{N}}$ be a non-increasing, positive sequence such that $\lim_{N \rightarrow \infty} W(N) = \infty$, where $W(N) := \sum_{n=1}^N w(n)$.
 We say that a sequence $(x_n)_{n \in \mathbb{N}}$, where $x_n = (x_{1,n}, \dots, x_{k,n})$, is $w(n)$-uniformly distributed $\bmod \, 1$ in $\mathbb{R}^{k}$ if for every Riemann integrable function $F$ on $\mathbb{T}^k$, 
\begin{equation}
\label{eq:w-uniform:def}
\lim_{N \rightarrow \infty} \frac{1}{W(N)} \sum_{n=1}^N w(n) F( \{x_{1, n}\}, \dots, \{x_{k,n}\}) = \int_{\mathbb{T}^k} F \, d \lambda^{k}.
\end{equation}
It is well known that to establish $w(n)$-uniform distribution it suffices to check \eqref{eq:w-uniform:def} for continuous functions. In fact, by the weighted version of Weyl's criterion (see, for example, Theorem 6 in \cite{Tsuji}), in order to verify $w(n)$-uniform distribution, it suffices to check \eqref{eq:w-uniform:def} for exponential functions of the form $e (\langle h, x_n \rangle )$, where 
$\langle h, x_n \rangle = \sum\limits_{j=1}^k h_j x_{j,n}$ for every nonzero vector $h= (h_1, \dots, h_k) \in \mathbb{Z}^k$.

The next theorem provides a characterization of weighted uniform distribution for sequences arising from subpolynomial functions in a Hardy field. 
\begin{thm}[cf. Theorem 4.2 in \cite{BKS2}]
\label{thm:w-ud:mixedindex} 
Let $u_1(t), \dots, u_k(t)$ be subpolynomial functions in a Hardy field.
Let $W(t) \in \mathbf{E}$ be such that $\lim_{t \rightarrow \infty} W(t) = \infty$ and $w(t) := W'(t)$ is positive and non-increasing.
The following are equivalent.
\begin{enumerate}
\item $(u_1(n), \dots, u_k(n))_{n \in \mathbb{N}}$ is $w(n)$-uniformly distributed $\bmod \, 1$ in $\mathbb{R}^{k}$.
\item For any $u(t) \in \text{span}_{\mathbb{Z}}^* \{ u_1(t), \dots, u_k(t) \}$,
\[ \lim_{t \rightarrow \infty} \frac{|u(t) -q(t)|}{\log W(t)} = \infty \quad \text{for any } q(t) \in \mathbb{Q}[t].\] 
\end{enumerate}
\end{thm}

The following lemma is a consequence of the previous result.
\begin{lem}
\label{lem:w-ud}
Let $\alpha_1, \dots, \alpha_k$ be irrational numbers. 
Let $u_1(t), \dots, u_k(t)$ be subpolynomial functions in a Hardy field.
Let $W(t) \in \mathbf{E}$ be such that $\lim_{t \rightarrow \infty} W(t) = \infty$ and $w(t) := W'(t)$ is positive and non-increasing.
Assume that for any $u \in \text{span}_{\mathbb{Z}}^* \{u_1, \alpha_1 u_1, \dots, u_k, \alpha_k u_k\}$,
$$\lim\limits_{t \rightarrow \infty} \frac{|u(t) -q(t)|}{ \log W(t)}  =  \infty \quad \text{for any } q(t) \in \mathbb{Q} [t].$$
Then the sequence $(\alpha_1 [u_1(n)], \dots, \alpha_k [u_k(n)])$
is $w(n)$-uniformly distributed $\bmod \, 1$ in $\mathbb{R}^{k}$. 
\end{lem}
\begin{proof}
For $(h_1, \dots, h_k) \in \mathbb{Z}^k \setminus \{(0, \dots , 0)\}$, let $$F(x_1, \dots, x_k, y_1, \dots, y_k) = \prod\limits_{j=1}^k e (h_j (x_j - \alpha_j \{y_j\})).$$
Note that 
$$\int_{\mathbb{T}^{2k}} F(x_1, \dots, x_k, y_1, \dots, y_k) =  \prod_{j=1}^k \int_{\mathbb{T}^2} e (h_j (x_j - \alpha_j \{y_j\})) dx_j dy_j= 0.$$
By Theorem \ref{thm:w-ud:mixedindex} , $(\alpha_1 u_1(n), \dots, \alpha_k u_k(n), u_1(n), \dots, u_k(n))_{n \in \mathbb{N}}$ is $w(n)$-uniformly distributed $\bmod \, 1$ in $\mathbb{R}^{2k}$.
Hence, for any nonzero $(h_1, \dots, h_k) \in \mathbb{Z}^k$,
\begin{align*}
\lim_{N \rightarrow \infty} \frac{1}{W(N)}  &\sum_{n=1}^N w(n) e \left(\sum_{j=1}^k h_j \alpha_j [u_j(n)] \right) \\
&= \lim_{N \rightarrow \infty} \frac{1}{W(N)}  \sum_{n=1}^N w(n) e \left( \sum_{j=1}^k h_j \alpha_j u_j (n) - h_j \alpha_j \{u_j(n)\} \right) \\
\quad \quad &=\lim_{N \rightarrow \infty} \frac{1}{W(N)}  \sum_{n=1}^N w(n) F (\alpha_1 u_1(n), \dots, \alpha_k u_k(n), u_1(n), \dots, u_k(n) )\\
\quad \quad &= \int_{\mathbb{T}^{2k}} F(x_1, \dots, x_k, y_1, \dots, y_k)= 0. 
\end{align*}
Therefore, the sequence $(\alpha_1 [u_1(n)], \dots, \alpha_k [u_k(n)])$ is $w(n)$-uniformly distributed $\bmod \, 1$ in $\mathbb{R}^{k}$.
\end{proof}

The next result will be useful in the proofs of Theorems \ref{thm:multiconv:int} and \ref{thm:multiconv2:int}.
\begin{thm}
\label{multi-conv:thm:cond} 
Let $X$ be a compact metric space equipped with the Borel $\sigma$-algebra $\mathcal{B}$ and let $\mu$  be a probability measure on $(X, \mathcal{B})$.
Let $(S_n^{(1)})_{n \in \mathbb{N}}, \dots, (S_n^{(k)})_{n \in \mathbb{N}} $ be sequences of measure preserving transformations on $(X, \mathcal{B}, \mu)$. 
Suppose that 
\begin{enumerate}
\item For each $i = 1, 2, \dots, k$, for any $f \in L^{\infty}(\mu)$, 
\begin{equation}
\label{eq:a.e.conv:infty}
\lim_{N \rightarrow \infty} \frac{1}{W(N)} \sum_{n=1}^N w(n) f(S_n^{(i)} x) = \int f \, d \mu
\end{equation}
for $\mu$-almost every $x$.
\item For any continuous functions $g_1, \dots, g_k$ on $X$,
\begin{equation}
\label{eq:a.e.conv:cont}
 \lim_{N \rightarrow \infty} \frac{1}{W(N)} \sum_{n=1}^N w(n) \prod_{i=1}^k g_i (S_n^{(i)} x) =  \prod_{i=1}^k \int g_i \, d \mu
\end{equation}
for $\mu$-almost every $x$.
\end{enumerate}
Then, for any functions $f_1, f_2, \dots, f_k \in L^{\infty}(\mu)$, 
\begin{equation}
\label{rem1:eq} 
\lim_{N \rightarrow \infty} \frac{1}{W(N)} \sum_{n=1}^{N} w(n) \prod_{i=1}^k f_i (S_n^{(i)} x)  = \prod_{i=1}^k \int f_i \, d \mu,
\end{equation}
 for $\mu$-almost every $x$.
\end{thm}

\begin{proof}
Without loss of generality, we assume that $\| f_i \|_{L^{\infty}} \leq 1$ for all $1 \leq i \leq k$. 
Since the set $C(X)$ of continuous functions on $X$ is dense in $L^1(X)$, for each $m \in \mathbb{N}$, there exist $g_{i,m} \in C(X)$ for $1 \leq i \leq k$ such that 
\begin{equation}
\label{eq:approxi:quan}
\|g_{i,m} \|_{L^{\infty}} \leq 1 \text{ and } \| g_{i,m} - f_i \|_{L^1} \leq \frac{1}{m}.
\end{equation}
By \eqref{eq:a.e.conv:infty} and \eqref{eq:a.e.conv:cont}, there exists a full measure set $X_0 \subset X$ such that if $x \in X_0$, then for $1 \leq i \leq k$ and $m \in \mathbb{N}$,
\begin{equation}
\label{eq:a.e.conv:infty2}
\lim_{N \rightarrow \infty} \frac{1}{W(N)} \sum_{n=1}^N w(n) |f_i - g_{i,m}| (S_n^{(i)} x) = \int |f_i - g_{i,m}| \, d \mu
\end{equation}
and 
\begin{equation}
\label{eq:a.e.conv:cont2}
 \lim_{N \rightarrow \infty} \frac{1}{W(N)} \sum_{n=1}^N w(n) \prod_{i=1}^k g_{i,m} (S_n^{(i)} x) =  \prod_{i=1}^k \int g_{i,m} \, d \mu.
\end{equation}
Write 
\begin{align*}
\frac{1}{W(N)} \sum_{n=1}^{N} w(n) \prod_{i=1}^k f_i (S_n^{(i)} x)  -  \prod_{i=1}^k \int f_i \, d \mu 
= A^{(1)}_{m,N}(x)+A^{(2)}_{m,N}(x)+E(m),
\end{align*}
where
\begin{align*}
A^{(1)}_{m,N}(x)&:=\frac{1}{W(N)} \sum_{n=1}^{N} w(n) \left( \prod_{i=1}^k f_i (S_n^{(i)} x)  - \prod_{i=1}^k g_{i,m} (S_n^{(i)} x) \right),\\
 A^{(2)}_{m,N}(x)&:=\frac{1}{W(N)} \sum_{n=1}^{N} w(n) \prod_{i=1}^k g_{i,m} (S_n^{(i)} x) -  \prod_{i=1}^k \int g_{i,m} \, d \mu \ \  \text{ and } \\
E(m)&:= \prod_{i=1}^k \int g_{i,m} \, d \mu  -  \prod_{i=1}^k \int f_i \, d \mu. 
\end{align*}
Applying the identity
\begin{equation}
\label{identity:2.4}
    \prod_{i=1}^k a_i - \prod_{i=1}^k b_i = 
   (a_1 - b_1) b_2 \cdots b_k + a_1 (a_2 - b_2) b_3 \cdots b_k + \cdots + a_1 \cdots a_{k-1} (a_k - b_k),
\end{equation}
we get
$$ \left|A^{(1)}_{m,N}(x) \right| \leq 
\sum_{i=1}^k \frac{1}{W(N)} \sum_{n=1}^N w(n) |f_i -g_{i,m}| (S_n^{(i)} x). $$
Thus, by \eqref{eq:approxi:quan} and \eqref{eq:a.e.conv:infty2}, for $x \in X_0$,
\begin{equation}
\label{eq1:conclusion:multi-conv:thm}
\limsup_{N \rightarrow \infty}  \left| A^{(1)}_{m,N} (x)\right| \leq  \frac{k}{m}.
\end{equation}
By \eqref{eq:a.e.conv:cont2}, for each $x \in X_0$,  
\begin{equation}
\label{eq2:conclusion:multi-conv:thm}
  A^{(2)}_{m,N}(x)  \rightarrow 0 \text{ as } N\to \infty.
\end{equation}
Applying the identity \eqref{identity:2.4} again, we obtain
\begin{equation}
\label{eq3:conclusion:multi-conv:thm}
\left|  E(m) \right| \leq \sum_{i=1}^k \int |g_{i,m} - f_i| \leq \frac{k}{m}.
\end{equation}
Combining \eqref{eq1:conclusion:multi-conv:thm} - \eqref{eq3:conclusion:multi-conv:thm}, for every $x\in X_0$ we obtain 
\[ \limsup_{N \rightarrow \infty} \left|  \frac{1}{W(N)} \sum_{n=1}^{N} w(n) \prod_{i=1}^k f_i (S_n^{(i)} x)  -  \prod_{i=1}^k \int f_i \, d \mu  \right| \leq \frac{2k}{m}.\]
Since $m$ is arbitrary, \eqref{rem1:eq}  follows.
\end{proof}

We now restate Theorem \ref{thm:multiconv:int} and give its proof.
\begin{thm}[Theorem \ref{thm:multiconv:int}]
\label{thm:multiconv}
Let $W(t) \in \mathbf{E}$ be such that $w(t):= W'(t)$ is positive and non-increasing and $\lim_{t \rightarrow \infty} W(t) = \infty$.
Let $(p_n)_{n \in \mathbb{N}}$ be the sequence of prime numbers.
Let $a_1(t), \dots, a_l(t), b_1(t), \dots, b_m(t) \in \mathbf{U}$, and let $\alpha_1, \dots, \alpha_l, \beta_1, \dots, \beta_m \in \mathbb{R}$ be irrational numbers.
Define
$$\mathcal{F} = \{ c_i a_i(t): c_i = 1 \text{ or } \alpha_i, \, 1 \leq i \leq l\} \cup \{ d_j b_j(t \log t): d_j = 1 \text{ or } \beta_j, \, 1 \leq j \leq m \}.$$
Assume that 
\begin{enumerate}[(i)]
\item for all $i=1, 2, \dots, l$, $1 \prec a_i(t) \prec t$, 
\item for all $j=1, 2, \dots, m$, $1 \prec b_j(t) \prec \log t$, 
\item $a_1(t), \dots, a_l(t), b_1(t \log t), \dots, b_m(t \log t) \in \mathbf{H}$ for some Hardy field $\mathbf{H}$ and 
$$\log W(t) \ll \log a_i(t) \text{ and } \log W(t) \ll \log b_j(t) \text{ for all } 1 \leq i \leq l, 1 \leq j \leq m,$$ 
\item for any $u(t) \in \text{span}_{\mathbb{Z}}^* \, \mathcal{F}$,
\begin{equation}
\label{eq:ud:cond}
\lim_{t \rightarrow \infty}  \frac{|u(t)|}{\log W(t)} = \infty. 
\end{equation}
\end{enumerate}

Then for any bounded measurable functions $f_1, \dots, f_l, g_1, \dots, g_m$ on $\mathbb{T} \, (:= \mathbb{R} / \mathbb{Z})$, 
\begin{equation*}
\label{eq3.4}
\frac{1}{W(N)} \sum_{n=1}^N w(n) \prod_{i=1}^l f_i (x + \lfloor a_i(n) \rfloor \alpha_i) \prod_{j=1}^m g_j( x+ \lfloor b_j(p_n) \rfloor \beta_j) \rightarrow \prod_{i=1}^l \int f_i \, d\lambda \prod_{j=1}^m \int g_j \, d \lambda
\end{equation*}  
for $\lambda$-almost every $x$.
\end{thm}
Before proving the theorem, we first state an auxiliary lemma that will be used to handle the prime sequences appearing in the theorem. It shows that if $u(t)$ is a slowly growing function in a Hardy field, then $\lfloor u(p_n) \rfloor$ and $\lfloor u(n\log n) \rfloor$ differ only on a set of indices whose contribution to the weighted average is asymptotically negligible.
\begin{lem}[cf. Lemma 4.8 in \cite{BKS2}]
\label{lem:comp:slow:prime}
Suppose $u(t) \in {\mathbf U}$ satisfies
\begin{equation*}
\label{eq:lem3.8}
(1) \lim\limits_{t \rightarrow \infty} \frac{u(t)}{\log t} < \infty \text{ and } (2) \lim\limits_{t \rightarrow \infty} \frac{u(t)}{ \log W(t)} = \infty.
\end{equation*}
Define $D := \{ n \in \mathbb{N} : \lfloor u(n \log n) \rfloor \ne \lfloor u(p_n) \rfloor \}.$
Then, 
\[ \lim_{N \rightarrow \infty} \frac{1}{W(N)} \sum_{n=1}^N w(n) \setone_{D} (n) = 0.\]
\end{lem}

\begin{proof}[Proof of Theorem \ref{thm:multiconv}.]
For $i = 1, 2, \dots, l+m$, 
define $T_i: \mathbb{T} \rightarrow \mathbb{T}$ by
\[T_i  x = 
\begin{cases}
x + \alpha_i,     & \text{ if } 1 \leq i \leq l, \\
x + \beta_{i-l},  & \text{ if } l+1 \leq i \leq l+m.
\end{cases}
\]
Let $S_n^{(i)} = T_i^{\lfloor a_i(n) \rfloor }$ for $ 1 \leq i \leq l $ and $S_n^{(i)} = T_{i}^{\lfloor b_{i-l} (p_n) \rfloor }$ for $ l+1 \leq i \leq l +m$.  

By Theorem \ref{thm:sufficient}, conditions (i) and (iii) imply that for $1 \leq i \leq l$, for any $f \in L^{\infty}(\lambda)$, 
\begin{equation*}
\lim_{N \rightarrow \infty} \frac{1}{W(N)} \sum_{n=1}^N w(n) f(S_n^{(i)} x) = \int f \, d \lambda
\end{equation*}
for $\lambda$-almost every $x$.
Note also that conditions (ii) and (iii) imply that $\log W(t) \ll \log b_j (t \log t)$ for $1 \leq j \leq m$, and it follows from Theorem \ref{thm:sufficient} that
for any $f \in L^{\infty}(\lambda)$, 
\begin{equation*}
\lim_{N \rightarrow \infty} \frac{1}{W(N)} \sum_{n=1}^N w(n) f(T_{j+l}^{\lfloor b_j (n \log n) \rfloor} x) = \int f \, d \lambda.
\end{equation*}
By Lemma \ref{lem:comp:slow:prime},  for $l+1 \leq i \leq l+m$, for any $f \in L^{\infty}(\lambda)$, 
\begin{equation*}
\lim_{N \rightarrow \infty} \frac{1}{W(N)} \sum_{n=1}^N w(n) f(S_n^{(i)} x) = \int f \, d \lambda
\end{equation*}
for $\lambda$-almost every $x$.
So, $(S_n^{(1)})_{n \in \mathbb{N}}, \dots, (S_n^{(l+m)})_{n \in \mathbb{N}}$ satisfy  condition (1) of Theorem \ref{multi-conv:thm:cond}. 

Therefore, it remains to verify that $(S_n^{(1)})_{n \in \mathbb{N}}, \dots, (S_n^{(l+m)})_{n \in \mathbb{N}}$ satisfy condition (2) of Theorem \ref{multi-conv:thm:cond}.
We will show that for any continuous $1$-periodic functions $f_1, \dots, f_l,$ $ g_1, \dots, g_m$ and for all $x$,
\begin{equation}
\label{eq:udcheck:thm}
\lim_{N \rightarrow \infty} \frac{1}{W(N)} \sum_{n=1}^N w(n) \prod_{i=1}^l f_i (x + \lfloor a_i(n) \rfloor \alpha_i) \prod_{j=1}^m g_j( x+ \lfloor b_j(p_n) \rfloor \beta_j) = \prod_{i=1}^l \int f_i \, d \lambda \prod_{j=1}^m \int g_j \, d \lambda.
\end{equation} 

Since $a_i(t) \prec t$ and $b_j(t\log t) \prec \log(t\log t) \prec t$, every $u(t) \in \operatorname{span}_{\mathbb{Z}}^* \mathcal{F}$ satisfies $u(t)\prec t$. Hence, condition (iv) implies that
$$
\lim_{t\to\infty}\frac{|u(t)-q(t)|}{\log W(t)}=\infty
\quad\text{ for every }q(t)\in\mathbb{Q}[t].
$$
By Lemma \ref{lem:w-ud}, the sequence   
\[ ( \alpha_1 \lfloor a_1(n) \rfloor, \dots, \alpha_l \lfloor a_l(n) \rfloor,  \beta_1 \lfloor b_1(n \log n) \rfloor,  \dots, \beta_m \lfloor b_m(n \log n) \rfloor)  \]
is $w(n)$-uniformly distributed $\bmod \, 1$ in $\mathbb{R}^{l+ m}$. 
Thus, for any continuous $1$-periodic functions $f_1, \dots, f_l, g_1, \dots, g_m$ and for all $x$,
\begin{equation*}
\lim_{N \rightarrow \infty} \frac{1}{W(N)} \sum_{n=1}^N w(n) \prod_{i=1}^l f_i (x + \lfloor a_i(n) \rfloor \alpha_i) \prod_{j=1}^m g_j( x+ \lfloor b_j(n \log n) \rfloor \beta_j) = \prod_{i=1}^l \int f_i \, d \lambda \prod_{j=1}^m \int g_j \, d \lambda.
\end{equation*}  
Applying Lemma \ref{lem:comp:slow:prime} to each $b_j$, the set of indices $n$ for which
$$
\lfloor b_j(p_n)\rfloor\neq \lfloor b_j(n\log n)\rfloor
$$
for at least one $1\leq j\leq m$ has weighted density zero. Since the functions $g_1,\dots,g_m$ are bounded, \eqref{eq:udcheck:thm} follows.
This completes the proof.
\end{proof}

We conclude this section by providing a proof of Theorem \ref{thm:multiconv2:int}.
\begin{thm}[Theorem \ref{thm:multiconv2:int}]
\label{thm:multiconv2}
Let $a_1(t), \dots, a_l(t)$ be subpolynomial functions in a Hardy field. 
Let $\alpha_1, \dots, \alpha_l$ be irrational numbers.
Define 
$$\mathcal{F} = \{ a_1(t), \dots, a_l(t),  \alpha_1 a_1(t), \dots, \alpha_l a_l(t) \}.$$
For $i= 1, 2, \dots, l$, let $a_i(t) = p_i(t) +r_i(t)$ be the canonical decomposition. 
\begin{enumerate}[(1)]
\item Assume that for some $\eta > 0$, $|r_i(t)| \gg t^{\eta}$ for all $i$, and for any $u(t) \in \text{span}_{\mathbb{Z}}^* \, \mathcal{F}$,
\begin{equation}
\label{eq:ud:je:last}
\lim_{t \rightarrow \infty}  \frac{|u(t) - q(t)|}{\log t} = \infty \quad \text{for any } q(t) \in \mathbb{Q}[t]. 
\end{equation}
Then, for any bounded measurable functions $f_1, \dots, f_l$ on $\mathbb{T}$, 
\begin{equation*}
\frac{1}{N} \sum_{n=1}^N \prod_{i=1}^l f_i (x + \lfloor a_i(n) \rfloor \alpha_i) \rightarrow \prod_{i=1}^l \int f_i \, d \lambda
\end{equation*}  
for $\lambda$-almost every $x$.
\item Assume that for some $\eta > 1$, $|r_i(t)| \gg (\log t)^{\eta}$ for all $i$, and  for any $u(t) \in \text{span}_{\mathbb{Z}}^* \, \mathcal{F} $,
\begin{equation*}
\lim_{t \rightarrow \infty}  \frac{|u(t) -q(t)|}{ \log \log t} = \infty \quad \text{for any } q(t) \in \mathbb{Q}[t]. 
\end{equation*}
Then, for any bounded measurable functions $f_1, \dots, f_l$ on $\mathbb{T}$, 
\begin{equation*}
\frac{1}{\log N} \sum_{n=1}^N \frac{1}{n} \prod_{i=1}^l f_i (x + \lfloor a_i(n) \rfloor \alpha_i) \rightarrow \prod_{i=1}^l \int f_i \, d \lambda 
\end{equation*}  
for $\lambda$-almost every $x$.
\end{enumerate}
\end{thm}

\begin{proof}
Let us prove (1). (The proof for (2) is analogous and omitted.)
For $i =1, 2, \dots, l$, define $T_i: \mathbb{T} \rightarrow \mathbb{T}$ by 
\[T_i x = x + \alpha_i.\]
Since $|r_i(t)| \gg t^{\eta}$ for some $\eta > 0$, it follows from Theorem \ref{sparse2:intro} that for each $1 \leq i \leq l$ and for any $f \in L^{\infty}$ 
\begin{equation*}
\lim_{N \rightarrow \infty} \frac{1}{N} \sum_{n=1}^N  f (T_i^{\lfloor a_i(n) \rfloor} x) = \int f \, d \lambda 
\end{equation*}
for $\lambda$-almost every $x$. 
By \eqref{eq:ud:je:last}, $( \alpha_1 \lfloor a_1(n) \rfloor, \dots, \alpha_l \lfloor a_l(n) \rfloor)_{n \in \mathbb{N}}$ is uniformly distributed $\bmod \, 1$, so for any continuous functions $g_1, \dots, g_l$ on $\mathbb{T} \, (:= \mathbb{R} / \mathbb{Z})$ and for all $x \in \mathbb{T}$
\begin{equation*}
\frac{1}{N} \sum_{n=1}^N \prod_{i=1}^l g_i (x + \lfloor a_i(n) \rfloor \alpha_i) \rightarrow \prod_{i=1}^l \int g_i(t) \, dt 
\end{equation*}  
The conclusion now follows from Theorem \ref{multi-conv:thm:cond}. 
\end{proof}

\section*{Acknowledgments}
Sovanlal Mondal would like to thank The Ohio State University, as most of the work presented in this paper was completed while he was affiliated with the university. Younghwan Son is supported by the Basic Science Research Institute Fund (NRF grant no. RS-2021-NR060139) and by the National Research Foundation of Korea (NRF) grant funded by the Korea government (MSIT) (RS-2025-24523340).

\printbibliography
\end{document}